\documentclass[aihp]{imsart}

\RequirePackage{amsthm,amsmath,amsfonts,amssymb}
\RequirePackage[numbers]{natbib}
\RequirePackage[colorlinks,citecolor=blue,urlcolor=blue]{hyperref}
\usepackage{nicefrac}
\usepackage{bbm}
\usepackage{xcolor}
\usepackage{comment}

\startlocaldefs

\newcommand{\Mult}{\text{\rm Mult}}

\newcommand{\be}{{\boldsymbol e}}
\newcommand{\bo}{{\boldsymbol O}}
\newcommand{\bp}{{\boldsymbol p}}
\newcommand{\bq}{{\boldsymbol q}}

\newcommand{\bu}{{\boldsymbol u}}

\newcommand{\bV}{{\boldsymbol V}}

\newcommand{\bW}{{\boldsymbol W}}
\newcommand{\bx}{{\boldsymbol X}}
\newcommand{\bxtilde}{\widetilde\bx}
\newcommand{\by}{{\boldsymbol Y}}

\newcommand{\bpi}{{\boldsymbol \pi}}
\newcommand{\bgamma}{{\boldsymbol \gamma}}
\newcommand{\btheta}{{\boldsymbol \theta}}
\newcommand{\bxi}{{\boldsymbol \xi}}
\newcommand{\bfeta}{{\boldsymbol \eta}}
\newcommand{\bone}{{\boldsymbol 1}}
\newcommand{\bzero}{{\boldsymbol 0}}

\newcommand{\dd}{{\rm d}}

\newcommand{\R}{\mathbb {R}}
\newcommand{\E}{\mathbb{E}}
\newcommand{\p}{\mathbb{P}} 
\newcommand{\1}{\mathbbm{1}} 

\newcommand{\B}{{\mathcal B}}
\newcommand{\C}{{\mathcal C}}

\newcommand{\F}{{\mathcal F}}
\newcommand{\ch}{{\mathcal H}}
\newcommand{\J}{{\mathcal J}}
\newcommand{\N}{{\mathcal N}}

\newcommand{\cq}{\mathcal{Q}}
\newcommand{\rr}{\mathcal{R}}
\newcommand{\cs}{\mathcal{S}}
\newcommand{\ct}{\mathcal{T}}

\newcommand{\ls}{\mathsf{L}}
\newcommand{\ps}{\mathsf{p}}

\newcommand{\z}{\mathfrak{z}}
\newcommand{\bz}{\boldsymbol{\z}}

\newcommand{\cond}{\,\Big\vert\, \btheta}
\newcommand{\law}[1]{\text{Law}\left(#1\right)}

\renewcommand{\leq}{\leqslant}
\renewcommand{\geq}{\geqslant}
\newcommand{\eqd}{\stackrel{d}{=}}

\numberwithin{equation}{section}

\theoremstyle{plain}
\newtheorem{Th}{Theorem}[section]
\newtheorem{Co}[Th]{Corollary}
\newtheorem{Lem}[Th]{Lemma}
\newtheorem{Prop}[Th]{Proposition}

\theoremstyle{remark}
\newtheorem{Rem}[Th]{Remark}
\newtheorem{Ex}[Th]{Example}
\newtheorem{As}[Th]{Assumption}

\endlocaldefs

\begin{document}

\begin{frontmatter}
\title{Inference via Randomized Test Statistics}
\runtitle{Inference via Randomized Test Statistics}

\begin{aug}
\author[A, B]{\inits{N.}\fnms{Nikita} \snm{Puchkin}\ead[label=e1,mark]{npuchkin@hse.ru}}
\and
\author[A, C]{\inits{V.}\fnms{Vladimir} \snm{Ulyanov}\ead[label=e2,mark]{vulyanov@cs.msu.ru}}

\address[A]{HSE University, Moscow, Russian Federation.
\printead{e1}}

\address[B]{Institute for Information Transmission Problems RAS, Moscow, Russian Federation}

\address[C]{Lomonosov Moscow State University, Moscow, Russian Federation.
\printead{e2}}

\end{aug}

\begin{abstract}
	We show that external randomization may enforce the convergence of test statistics to their limiting distributions in particular cases. This results in a sharper inference. Our approach is based on a central limit theorem for weighted sums. We apply our method to a family of rank-based test statistics and a family of phi-divergence test statistics and prove that, with overwhelming probability with respect to the external randomization, the randomized statistics converge at the rate $O(1/n)$ (up to some logarithmic factors) to the limiting chi-square distribution in Kolmogorov metric.
\end{abstract}

\begin{abstract}[language=french]
    Nous montrons que l’ajout de randomisation externe peut guider la convergence de statistiques de test vers leurs distributions limites dans certains cas particuliers. Il en r\'{e}sulte une inf\'{e}rence plus pr\'{e}cise. Notre approche est bas\'{e}e sur un th\'{e}or\`{e}me central limite pour les sommes pond\'{e}r\'{e}es. Nous appliquons notre m\'{e}thode \`{a} une famille de statistiques de test bas\'{e}es sur les rangs et \`{a} une famille de statistiques de test de phi-divergence et nous prouvons que, avec grande probabilit\'{e} par rapport \`{a} la randomisation externe, les statistiques randomis\'{e}es convergent \`{a} un taux $O(1/n)$ (n\'{e}gligeant certains facteurs logarithmiques) vers la distribution limite du chi-deux dans la distance de Kolomogorov.
\end{abstract}



\end{frontmatter}

\section{Introduction}

The present paper studies a family of statistics commonly used in hypothesis testing. We show that external randomization may speed up the convergence rate of test statistics to their limiting distributions in particular cases. To illustrate the importance of such a refinement, consider the following example. Let $\by = (Y_1, \dots, Y_r)$ be a $r$-dimensional random vector with a multinomial distribution $\Mult(n, \bp_\by)$. Assume that a statistician is interested in testing the null hypothesis $H_0 : \bp_\by = \bp$ where $\bp = (p_1, \dots, p_r)$ is a given probability vector. He can use Pearson's criterion \citep{pearson900} to perform this task. Namely, the statistician can compute the test statistic
\[
	T_P = \sum\limits_{j=1}^r \frac{(Y_j - n p_j)^2}{n p_j}
\]
and reject $H_0$ if and only if $T_P$ exceeds a critical value. Given a significance level $\alpha \in (0, 1)$, the critical value is usually computed as the $(1 - \alpha)$-quantile of the limiting distribution of $T_P$, $\chi^2(r - 1)$.
However, since the statistician deals with a finite sample, the distribution of $T_P$ slightly differs from $\chi^2(r - 1)$, and the type-I error of the criterion may exceed $\alpha$. Fortunately, even in the worst case, it is not greater than
\[
	\alpha + d_K(\law{T_P}, \chi^2(r-1)),
\]
where $d_K$ stands for the Kolmogorov distance. For any random variables $\xi$ and $\eta$, it is defined as
\[
    d_K(\law\xi, \law\eta) = \sup\limits_{t \in \R} \left| \p\left( \xi > t \right) - \p\left( \eta > t \right) \right|.
\]
If the statistician finds a statistic converging to its limiting distribution faster than $T_P$, he will get a tighter bound on the false discovery rate and, as consequence, will perform a more accurate inference.

Instead of Pearson's statistic, one could use, for instance, the likelihood ratio, or Freeman-Tukey's \citep{ft50}, or Cressie-Read's \citep{cr84} statistics. They all belong to a large family of phi-divergence test statistics we study in the present paper (see Section 2.2 for the definition). Besides, we consider a class of statistics based on $n$ rankings \citep{sen68}. The common property of the statistics of interest is that their asymptotic distribution is $\chi^2(r-1)$. Particular examples of the statistics we consider have been broadly investigated in the literature. State-of-the-art results quantify their rate of convergence to the chi-squared distribution for finite sample size $n$. Usually, researchers use integral probability metrics (IPMs) for this purpose. An integral probability metric between distributions of random elements $\xi$ and $\eta$ is defined as
\[
    d_\ch(\law\xi, \law\eta) = \sup\limits_{h \in \ch} \left| \E h(\xi) - \E h(\eta) \right|,
\]
where $\ch$ is a given class of test functions. One often chooses $\ch$ as a class of smooth functions or thresholds
\[
    \ch_K = \left\{ h_t(x) = \1(x > t) : t \in \R \right\}.
\]
The IPM, corresponding to $\ch_K$ is nothing but the Kolmogorov distance $d_K$.

In a recent paper \citep{gaunt21}, the author has shown that if $\ch$ is a class of sufficiently smooth functions, then the IPM between a power divergence test statistic (defined in Section \ref{sec:phi-div} below) and the corresponding chi-squared distribution decays as fast as $O(1/n)$.
A similar result was obtained in \citep{gr21} for Friedman's statistic. Unfortunately, the situation is different if one deals with the Kolmogorov distance. In this case, state-of-the-art results cannot achieve the rate of convergence $O(1/n)$ with rare exceptions (see an overview in Section \ref{sec:related_work}). Instead of refining the existing bounds, we suggest modifications of the phi-divergence and rank-based test statistics, adding external randomization, which enforces the convergence to the limiting distribution. Under very mild assumptions, we show that the Kolmogorov distance between the obtained randomized statistics and a chi-squared distribution tends to zero almost as fast as $O(1/n)$ (up to logarithmic factors) with overwhelming probability over the external randomization.

Our approach is based on a central limit theorem for weighted sums. In \citep{ks12}, the authors noticed that if one takes centered i.i.d. random variables $\xi_1, \dots, \xi_n$ with unit variance and a finite fourth moment and a vector of coefficients $\btheta \in \R^n$ drawn from the uniform distribution over the unit sphere $\cs^{n-1}$, then, for any $\delta \in (0, 1)$, with probability at least $1 - \delta$, it holds that
\begin{equation}
    \label{eq:ks12}
	\sup\limits_{\substack{a, b \in \R,\\ a < b}} \left| \p\left( a \leq \sum\limits_{i=1}^n \theta_i \xi_i \leq b \cond \right) - \frac1{\sqrt{2\pi}} \int\limits_a^b e^{-x^2/2} \dd x \right|
	\leq \frac{C_{KS} \E \xi_1^4 \log^2(1 / \delta)}{n},
\end{equation}
where $C_{KS}$ is an absolute constant. This result drastically improves the standard Berry-Esseen bound $O(1/\sqrt{n})$. Properties of weighted sums of random variables were further studied in \citep{bcg18, bcg20, bg20, bobkov20}. Recently, Ayvazyan and Ulyanov \cite{au20} extended the inequality \eqref{eq:ks12} to a multivariate case.

\begin{Th}[Ayvazyan and Ulyanov \cite{au20}]
	\label{th:au20}
	Let $\bxi_1, \dots, \bxi_n$ be i.i.d. random vectors in $\R^d$, $\E \bxi_1 = \bzero$, $\E \bxi_1 \bxi_1^\top = I_d$, $\E \|\bxi_1\|^4 < \infty$.
	Denote the family of convex Borel sets in $\R^d$ by $\mathfrak B$ and let $\bfeta \sim \N(\bzero, I_d)$ be the standard Gaussian random vector in $\R^d$. 
	Then, for any $\delta \in (0, 1)$, with probability at least $1 - \delta$ over $\btheta \sim \mathcal U(\cs^{n-1})$, it holds that
	\begin{equation}
	    \label{eq:au20}
		\sup\limits_{B \in \mathfrak B} \left| \p\left( \sum\limits_{i=1}^n \theta_i \bxi_i \in B \cond \right) - \p( \bfeta \in B ) \right|
		\leq \frac{C_d \E \|\bxi_1\|^4 \log^2(1 / \delta)}{n},
	\end{equation}
	where the constant $C_d$ depends on $d$ only.
\end{Th}
In the present work, we use Theorem \ref{th:au20} to justify convergence rates of the randomized rank-based and phi-divergence test statistics.

The rest of the paper is organized as follows.
In Section \ref{sec:related_work}, we introduce some basic definitions and overview the existing results.
In Section \ref{sec:randomized_rankings} and Section \ref{sec:randomized_phi-div}, we introduce randomized counterparts of some rank-based and phi-divergence test statistics, respectively, and quantify the rates of convergence of the randomized statistics to their limiting distributions in Theorem \ref{th:rankings} and Theorem \ref{th:phi-div}. Section \ref{sec:proofs} contains the proofs of our main results.

\subsection*{Notation}

For a set $\mathsf X \subset \R$, we denote a class of $k$ times differentiable functions with a bounded on $\mathsf X$ $k$-th derivative by $\C_b^k(\mathsf X)$.
The notation $\C_b^{j,k}(\mathsf X)$, $j < k$, stands for the class of functions from $\C_b^k(\mathsf X)$ with bounded on $\mathsf X$ derivatives of orders $j, j+1, \dots, k$.
We reserve a bold font for vectors while matrices and scalars are written in regular font.
Let $\alpha, \beta, \gamma$ be some parameters. Along with the standard $O(\cdot)$ notation, we use $O_{\alpha, \beta, \gamma}(\cdot)$ to emphasize that the hidden constant depends on $\alpha, \beta, \gamma$.
Sometimes we specify the dependence on parameters in the text if the notation becomes cumbersome.
Throughout the paper, the notation $f(n) \lesssim g(n)$ means that there exists a universal constant $c > 0$ such that $f(n) \leq cg(n)$ for all $n \in \mathbb N$.

\section{Related work}
\label{sec:related_work}

\subsection{Test statistics based on n rankings}

Let $\{ O_{ij} : 1 \leq i \leq n, \, 1 \leq j \leq r\}$ be independent random objects. For each $i$ from $1$ to $n$, let $\bo_i = (O_{i1}, \dots, O_{ir})$ be a tuple of size $r$ such that its components are distributed according to continuous cumulative distribution functions $F_{i1}, \dots, F_{ir}$. A statistician observes $n$ rankings $\{\bpi_i : 1 \leq i \leq n\}$ where $\bpi_i$ is a ranking of components of $\bo_i$. We treat each $\bpi_i = (\pi_{i1}, \dots, \pi_{ir})$ as a permutation of $\{1, \dots, r \}$. The random objects $O_{ij}$'s themselves may not be available to the statistician. The method of $n$ rankings \cite{sen68} is used to test the null hypothesis
\[
	H_0 : F_{i1} \equiv F_{i2} \equiv \dots \equiv F_{ir} \equiv F_i
	\quad \text{for all $i \in \{1, \dots, n\}$}.
\]
It is clear that, under $H_0$, the rankings $\bpi_1, \dots, \bpi_n$ are i.i.d. random permutations of $r$ elements with a uniform distribution.

To test the null hypothesis, Sen \cite{sen68} proposed the following family of statistics. Fix a function $\J : \{1, \dots r\} \rightarrow \R$ and introduce
\begin{equation}
	\label{eq:overline_j}
	\overline \J = \frac1r \sum\limits_{k=1}^r \J(k),
	\qquad
	\sigma^2_\J = \frac1{r - 1} \sum\limits_{k=1}^r \left( \J(k) - \overline\J \right)^2.
\end{equation}
For each $i \in \{1, \dots, n\}$, denote
\begin{equation}
	\label{eq:v}
	\bV_i^\J = \left( \J(\pi_{i1}), \dots, \J(\pi_{ir}) \right)^\top \in \R^r.
\end{equation}
The test statistic of the method of $n$ rankings, associated with $\J$, is given by
\begin{equation}
	\label{eq:tj}
	T_\J = \frac{1}{\sigma^2_\J n} \left\| \sum\limits_{i=1}^n \left( \bV_i^\J - \overline\J \bone \right) \right\|^2,
\end{equation}
where $\bone = (1, \dots, 1)^\top \in \R^r$. We proceed with two examples corresponding to particular choices of the function $\J$.

\begin{Ex}[Friedman's test statistic]
	Friedman's test statistic \citep{friedman37} corresponds to $\J(k) \equiv k$. In this case,
	\[
		\overline\J = \frac{r + 1}{2}
		\quad \text{and} \quad
		\sigma^2_\J = \frac{r(r+1)}{12}.
	\]
	Friedman's statistic is then defined as
	\[
		T_F = \frac{12}{r(r+1)n} \left\| \sum\limits_{i=1}^n \left( \bpi_i - \frac{r+1}2 \bone \right) \right\|^2.
	\]
\end{Ex}

\begin{Ex}[Brown-Mood's test statistic]
	If we fix $a \in \{1, \dots, r-1\}$ and take $\J(k) = \1(k \leq a)$, we obtain Brown-Mood's test statistic \citep{bm51}. It is straightforward to compute
	\[
		\overline\J = \frac ar
		\quad \text{and} \quad
		\sigma^2_\J = \frac{a(r - a)}{r(r - 1)},
	\]
	and then Brown-Mood's statistic is given by
	\[
		T_{BM} = \frac{r(r - 1)}{a(r - a)n} \left\| \sum\limits_{i=1}^n \left( \bV_i^{BM} - \frac ar \bone \right) \right\|^2,
	\]
	where $\bV_i^{BM} = \left( \1(\pi_{i1} \leq a), \dots, \1(\pi_{ir} \leq a) \right)^\top$. The statistic $T_{BM}$ was proposed in \cite{bm51} as an alternative to Friedman's statistic.
\end{Ex}

There is not much literature concerning the properties of $T_\J$. In \citep{sen68}, the author proved that $T_\J$ is asymptotically distributed as $\chi^2(r-1)$. One of the most famous examples of $T_\J$, Friedman's statistic $T_F$, is studied a bit better. In \citep{jensen77}, using \citep[Chapter 7, Theorem 1]{esseen45}, the author argued that the Kolmogorov distance between the distribution of $T_F$ and $\chi^2(r-1)$ decays as $O_r(n^{-1 + 1/r})$.  
The notation $O_r$ emphasizes that the hidden constant depends on $r$. To our knowledge, there are no similar bounds for $d_K(\law{T_\J}, \chi^2(r-1))$ in the literature. Fortunately, one can derive even better rates of convergence than $O_r(n^{-1 + 1/r})$ using recent advances in CLT for quadratic forms \cite[Theorem 1.2]{gz14} in addition to the classical result \cite[Chapter 7, Theorem 1]{esseen45}. We provide the details in Proposition \ref{prop:rankings} below. Its proof is postponed to Appendix \ref{app:prop_proof}.

\begin{Prop}
	\label{prop:rankings}
	Let $T_\J$ be as defined in \eqref{eq:tj}.
	Assume that the function $\J : \{1, \dots, r\} \rightarrow \R$ is such that, for all $k$ from $1$ to $r$,
	\[
		\left| \J(k) - \overline\J \right| \leq B,
	\]
	where $\overline\J$ is defined in \eqref{eq:overline_j} and $B$ is a constant (possibly depending on $r$).
	Then it holds that
	\[
		\sup\limits_{t \in \R} \left| \p\left( \ct_\J > t \right) - \p( Z > t ) \right| =
		\begin{cases}
			O \left( B^3 n^{ -1 + 1/r} / \sigma_\J^3 \right), \quad \text{if $2 \leq r \leq 5$},\\
			O_r\left( B^2 n^{-1} / \sigma_\J^2 \right), \quad \text{if $r \geq 6$},
		\end{cases}
	\]
	where $Z \sim \chi^2(r-1)$ and $\sigma_\J^2$ is defined in \eqref{eq:overline_j}.
\end{Prop}

Another line of research devoted to the rank-based statistics uses Stein's method to derive rates of convergence to the chi-squared distribution in different IPMs. In the recent paper \citep{gr21}, the authors considered a class $\C^{1,3}_b(\R_+)$ of test functions with bounded derivatives of order $1, 2$, and $3$ (functions from $\C^{1,3}_b(\R_+)$ do not need to be bounded).
They proved that, for any $h \in \C^{1,3}_b(\R_+)$,
\[
    \E h(T_F) - \E_{Z \sim \chi^2(r-1)} h(Z)
    \lesssim \frac rn \left( \|h'\|_{L_\infty(\R_+)} + \frac rn \left( \|h''\|_{L_\infty(\R_+)} + \|h'''\|_{L_\infty(\R_+)}\right) \right).
\]
Moreover, based on the findings of \citep{gaunt20}, Gaunt and Reinert \cite[Proposition 1.5]{gr21} derived a $O(n^{-1/2})$ upper bound on the Wasserstein distance between $T_F$ and $\chi^2(1)$ for the case $r=2$.

\subsection{Phi-divergence test statistic}
\label{sec:phi-div}

Let $\bp = (p_1, \dots, p_r)$ be a $r$-dimensional probability vector and let $\by = (Y_1, \dots, Y_r)$ have a multinomial distribution $\Mult(n, \bp)$, that is, for any $n_1, \dots, n_r \in \{0, 1, \dots, n\}$ such that  $n_1 + \dots + n_r = n$, 
\[
	\p\left( Y_1 = n_1, Y_2 = n_2, \dots, Y_r = n_r \right)
	= \frac{n!}{n_1! n_2! \dots n_r!} p_1^{n_1} p_2^{n_2} \dots p_r^{n_r}.
\]
A phi-divergence test statistic is defined as (see, for instance, \cite[Section 3.4.1]{pardo05})
\begin{equation}
	\label{eq:phi-div}
	T_\phi = \frac{2n}{\phi''(1)} \sum\limits_{j=1}^r p_j \phi\left( \frac{Y_j}{n p_j} \right).
\end{equation}
Here $\phi$ is a convex non-negative function on $\R_+$ such that
\begin{align*}
	&
	\phi(1) = \phi'(1) = 0, \quad \phi''(1) > 0.
\end{align*}
\begin{Ex}
	An important example
	\[
	    \phi_\lambda(u) = \frac{u^{\lambda + 1} - (\lambda + 1) (u - 1) - 1}{\lambda(\lambda + 1)}
	\]
	corresponds to the family of power divergence test statistics \citep{cr84}:
	\[
	    T_{\phi_\lambda} = \frac{2}{\lambda(\lambda + 1)} \sum\limits_{j=1}^r Y_j \left[ \left( \frac{Y_j}{n p_j} \right)^\lambda - 1 \right].
	\]
    When $\lambda = -1$ or $\lambda = 0$, $\phi_\lambda(u)$ should be understood as a passage to limit, that is,
    \[
        \varphi_{-1}(u) = -\log u + u - 1
        \quad \text{and} \quad
        \varphi_0(u) = u\log u - u + 1.
    \]
	The values $\lambda = -1$, $\lambda = -1/2$, $\lambda = 0$, $\lambda = 2/3$, and $\lambda = 1$ correspond to the modified log-likelihood ratio, Freeman-Tukey, log-likelihood ratio, Cressie-Read, and Pearson test statistics, respectively.
\end{Ex}

Particular examples of phi-divergence test statistics were extensively studied in the literature.
For Pearson's statistic $T_{\phi_1}$, G{\" o}tze and Ulyanov \cite{gu03} used an expansion from \cite{yarnold72} and advances in lattice point problems \citep[Theorem 1.5]{gotze04} to prove the upper bound
\[
    d_K\left(\law{T_{\phi_1}}, \chi^2(r-1) \right)
    = \begin{cases}
        O_{p_{\min}}\left(n^{-1 + 1/r} \right), \quad \text{if $2 \leq r \leq 5$},\\
        O_{r, p_{\min}}\left( n^{-1} \right), \quad \text{if $r \geq 6$}.
    \end{cases}
\]
Note that the Kolmogorov distance between Pearson's statistic and the chi-squared distribution decays as $O(1/n)$ not for all $r \geq 2$. For instance, in the simplest case $r=2$ and $\bp = (1/2, 1/2)^\top$, Pearson's statistic becomes proportional to a squared centered binomial random variable $\text{Binom}(n, 1/2)$ and, consequently, has an atom with probability mass of order $n^{-1/2}$ at zero for all even $n$. This yields that the $O(1/\sqrt{n})$ rate of convergence cannot be improved for $r=2$. The question about the optimality of the upper bound $O(n^{-1 + 1/r})$ for $r \in \{3, 4, 5\}$ still remains open.

The approach of G{\" o}tze and Ulyanov \cite{gu03} was further developed in \cite{uz09} for the power divergence family of statistics and then refined in \cite{assylbekov10, azu11}:
\begin{equation}
	\label{eq:power_div_rates}
	d_K\left(\law{T_{\phi_\lambda}}, \chi^2(r-1) \right)
	= \begin{cases}
		O\left(n^{-50/73} (\log n)^{315/146} \right), \quad \text{if $r = 3$},\\
		O\left(n^{-1 + 6/(7r-3)} \right), \quad \text{if $4 \leq r \leq 8$},\\
		O\left(n^{-1 + 5/(6r-4)} \right), \quad \text{if $r \geq 9$}.
	\end{cases}
\end{equation}
Here $\lambda \in \R$ and the hidden constants in $O(\cdot)$ depend on $r, \lambda$, and $p_{\min}$.
The main obstacle in the analysis of $T_{\phi_\lambda}$, $\lambda \neq 1$, is that, in contrast to Pearson's statistic, it is not a quadratic form and, consequently, the results of G{\" o}tze \cite[Theorem 1.5]{gotze04} or of G{\" o}tze and Zaitsev \cite[Theorem 1.2]{gz14} are not applicable anymore. Due to Taylor's expansion, the phi-divergence test statistic $T_\phi$ can be represented in the form
\[
	T_\phi
	= \sum\limits_{j=1}^r \frac{(Y_j - np_j)^2}{np_j} + \frac{\phi'''(1)}{3\phi''(1)} \sum\limits_{j=1}^r \frac{(Y_j - np_j)^3}{n^2 p_j^2} + R,
\]
where $R$ is a remainder of order $O((\log n)^4 / n)$. It follows from \cite{gu03} that the quadratic term
\[
	\sum\limits_{j=1}^r \frac{(Y_j - n p_j)^2}{n p_j}
\]
convergences in distribution to $\chi^2(r-1)$ with the speed $O_{r, p_{\min}}(1/n)$, provided that $r \geq 6$. However, the cubic terms
\[
	\frac{\phi'''(1)}{3\phi''(1)} \sum\limits_{j=1}^r \frac{(Y_j - np_j)^3}{n^2 p_j^2}
\]
of order $O(n^{-1/2})$ worsen this rate of convergence as one can see in \eqref{eq:power_div_rates}. In Section \ref{sec:randomized_phi-div}, we introduce a randomized counterpart of $T_\phi$ to bypass this obstacle.

In \cite{gpr17} and \cite{gaunt21}, the authors used Stein's method to specify the rates of convergence of the Pearson and power divergence statistics to $\chi^2(r-1)$ in terms of integral probability metrics with $\ch = \C_b^2(\R_+)$ and $\ch = \C_b^5(\R_+)$.
For Pearson's statistic, Gaunt et al. \cite{gpr17} obtained
\[
    d_{\C_b^2(\R_+)}\left(\law{T_{\phi_1}}, \chi^2(r-1) \right) = O_r\left( \frac1{\sqrt{ n p_{\min}}} \right)
\]
and
\[
    d_{\C_b^5(\R_+)}\left(\law{T_{\phi_1}}, \chi^2(r-1) \right) = O_r\left( \frac1{n p_{\min}} \right)
\]
with explicit constants and dependence on $r$ and $p_{\min} = \min\{p_1, \dots, p_r\}$.
Similarly, Gaunt \cite{gaunt21} proved that
\[
    d_{\C_b^{1, 2}(\R_+)} \left(\law{T_{\phi_\lambda}}, \chi^2(r-1) \right) = O_{r, \lambda} \left( \frac1{\sqrt{ n p_{\min}}} \right)
\]
and
\[
    d_{\C_b^{1, 5}(\R_+)} \left(\law{T_{\phi_\lambda}}, \chi^2(r-1) \right) = O_{r, \lambda} \left( \frac1{n p_{\min}} \right)
\]
for slightly broader classes of test functions.
For a general function $\phi$, it is known that $T_\phi$ has an asymptotic distribution $\chi^2(r-1)$ (see, for example, \cite[Theorem 3.1]{pardo05}).
We have not found finite sample bounds on the Kolmogorov distance between the law of $T_\phi$ and $\chi^2(r-1)$ in the literature but, following the proof of Pardo \cite[Theorem 3.1]{pardo05} and applying the Berry-Esseen inequality and standard concentration inequalities, one can show that $d_K(\law{T_\phi}, \chi^2(r-1))$ decays as fast as $O_{r, p_{\min}}(n^{-1/2})$ under mild assumptions on the function $\phi$.

\section{Randomized method of n rankings}
\label{sec:randomized_rankings}

In this section, we illustrate our approach with a simple example of randomized method of $n$ rankings.
Let $\btheta$ be distributed uniformly on a unit sphere $\cs^{n-1}$.
Define a randomized counterpart of $T_\J$ (given by \eqref{eq:tj}) as follows:
\begin{equation}
	\label{eq:rand_tj}
	\ct_\J = \frac{1}{\sigma^2_\J} \left\| \sum\limits_{i=1}^n \theta_i \left( \bV_i^\J - \overline\J \bone \right) \right\|^2,
\end{equation}
where $\sigma^2_\J$, $\overline \J$, and $\bV_1^\J, \dots, \bV_n^\J$ were introduced in \eqref{eq:overline_j} and \eqref{eq:v}. For practical purposes, note that there is an easy way to sample from the uniform distribution on the sphere. If one has an isotropic Gaussian vector $\bgamma \sim \N(\bzero, I_n)$, then $\btheta$ has the same distribution as $\bgamma / \|\bgamma\|$.

\begin{Rem}
	Using the equalities
	\[
		\E \theta_j \theta_k
		= \begin{cases}
			\frac1n, \quad \text{if $j = k$,}\\
			0, \quad \text{otherwise,}
		\end{cases}
	\]
	it is easy to show that
	\[
		\E \left( \ct_\J \,\vert\, \bpi_1, \dots, \bpi_n \right) = T_\J
		\quad \text{almost surely,}
	\]
	that is, the averaging of the randomized statistic $\ct_\J$ with respect to external randomization results in the standard statistic $T_\J$.
\end{Rem}

The modified statistic $\ct_\J$ converges to the limiting distribution $\chi^2(r-1)$ as its counterpart $T_\J$ does.
Using Theorem \ref{th:au20}, it is easy to show that the Kolmogorov distance between $\ct_\J$ and $\chi^2(r-1)$ decays as fast as $O(1/n)$. We provide a rigorous statement in the following theorem.

\begin{Th}
	\label{th:rankings}
	Let $\ct_\J$ be as defined in \eqref{eq:rand_tj}.
	Assume that the function $\J : \{1, \dots, r\} \rightarrow \R$ is such that, for all $k$ from $1$ to $r$,
	\begin{equation}
		\label{eq:j_boundness}
		\left| \J(k) - \overline\J \right| \leq B
		\quad \text{almost surely},
	\end{equation}
	where $\overline\J$ is defined in \eqref{eq:overline_j} and $B$ is a constant (possibly depending on $r$).
	Then, for any $\delta \in (0, 1)$, with probability at least $1 - \delta$ (over $\btheta$ uniformly distributed on the 	unit sphere $\cs^{n-1}$), it holds that
	\begin{equation}
		\label{eq:rankings_bound}
		\sup\limits_{t \in \R} \left| \p\left( \ct_\J > t \cond \right)  - \p\left( Z > t \right) \right|
		\leq
		\frac{C_{r-1} r^2 B^2 \log^2(1 / \delta)}{\sigma_\J^2 n},
	\end{equation}
	where $Z \sim \chi^2(r-1)$, $C_{r-1}$ is the same constant as in \eqref{eq:au20}, and $\sigma_\J^2$ is defined in \eqref{eq:overline_j}.
\end{Th}

\begin{Rem}
The proof of Theorem \ref{th:rankings} uses even weaker result than the one of Theorem \ref{th:au20}.
Namely, instead of comparing the probabilities in \eqref{eq:au20} on all convex Borel sets, it would be enough to ensure that \eqref{eq:au20} holds for centered Euclidean balls only. 
\end{Rem}

The proof of Theorem \ref{th:rankings} is postponed to Section \ref{sec:proof_rankings}.
Before we prove Theorem \ref{th:rankings}, let us specify the right hand side of \eqref{eq:rankings_bound} for the case of Friedman's test statistic. If $\J(k) \equiv k$, then it is easy to check that \eqref{eq:j_boundness} holds with $B = (r+1) / 2$ and that $\sigma_\J^2 = r(r + 1) / 12$. Substituting these bounds into \eqref{eq:rankings_bound}, we get the following corollary.

\begin{Co}
	\label{co:friedman}
	Consider the randomized Friedman's test statistic
	\[
		\ct_F = \frac{12}{r(r+1)} \left\| \sum\limits_{i=1}^n \theta_i \left( \bpi_i - \frac{r+1}2 \bone \right) \right\|^2,
	\]
	where $\btheta$ is drawn from a uniform distribution on the unit sphere $\cs^{n-1}$. Then, for any $\delta \in (0, 1)$, with probability at least $1 - \delta$ (over realizations of $\btheta$), it holds that
	\[
		\sup\limits_{t \in \R} \left| \p\left( \ct_F > t \cond \right)  - \p\left( Z > t \right) \right|
		\leq
		\frac{36 C_{r-1} (r+1)^2\log^2(1 / \delta)}{n},
	\]
	where $Z \sim \chi^2(r-1)$ and $C_{r-1}$ is the same constant as in \eqref{eq:au20}.
\end{Co}

Corollary \ref{co:friedman} yields that the Kolmogorov distance between $\ct_F$ and its limiting distribution decays as fast as $O(1/n)$, which improves over the best known bound for the standard Friedman statistic (see \citep[Theorem 2]{jensen77} and Proposition \ref{prop:rankings}) in the case $2 \leq r \leq 5$.

\section{Randomized phi-divergence test statistics}
\label{sec:randomized_phi-div}

Let us describe the construction of the randomized phi-divergence test statistics.
Note that, the vector $\by \sim \Mult(n, \bp)$ can be represented as a sum of i.i.d. multinomial $\Mult(1, \bp)$ random vectors $\bfeta_1, \dots, \bfeta_n$:
\begin{equation}
    \label{eq:y_decomposition}
    \by = \sum\limits_{i=1}^n \bfeta_i.
\end{equation}
It is straightforward to check that
\begin{equation}
    \label{eq:sigma}
	\E \bfeta_1 = \bp, \quad \text{Var}(\bfeta_1) = D - \bp\bp^\top \equiv \Sigma,
\end{equation}
where $D = \text{diag}(\bp)$.
Let $\btheta $ have a uniform distribution on a unit sphere $\cs^{n-1}$.
Define the centered random vectors $\bxi_i = \bfeta_i - \bp$, $1 \leq i \leq n$, a weighted sum
\[
	\bx^\btheta = \sum\limits_{i=1}^n \theta_i \bxi_i.
\]
and introduce a randomized phi-divergence test statistic
\begin{equation}
	\label{eq:rand_phi-div}
	\ct_\phi = \frac{2n}{\phi''(1)} \sum\limits_{j=1}^r p_j \phi\left( 1 + 
	\frac{X_j^\btheta}{\sqrt{n} 
	p_j} \right).
\end{equation}
In the rest of the paper, we assume that the function $\phi$ in \eqref{eq:rand_phi-div} satisfies the following.
\begin{As}
	\label{as:phi}
	The function $\phi$ is three times differentiable at $1$,
	\[
		\phi(1) = \phi'(1) = 0, \qquad \phi''(1) > 0,
	\]
	and the third derivative $\phi'''$ is $\ls$-Lipschitz on $[1 - \Delta, 1 + \Delta]$ for some $\Delta > 0$.
\end{As}
Assumption \ref{as:phi} is quite mild. In particular, the functions $\phi_\lambda$, $\lambda \in \R$, used in power divergence statistics, satisfy this assumption. 
\begin{Prop}
    \label{prop:power}
    Fix $\lambda \in \R$ and consider a function
    \begin{equation}
        \label{eq:power_div}
        \phi_\lambda(u) =
        \begin{cases}
            -\log u + u - 1, \quad \text{if $\lambda = -1$},\\
            u\log u - u + 1, \quad \text{if $\lambda = 0$},\\
            \frac{u^{\lambda + 1} - (\lambda + 1)(u - 1) - 1}{\lambda(\lambda + 1)}, \quad \text{otherwise}.
        \end{cases}
    \end{equation}
    Then $\phi_\lambda$ satisfies Assumption \ref{as:phi} with the following constants.
    \begin{itemize}
        \item[(a)] If $\lambda = 1$ or $\lambda = 2$, then $\phi_\lambda$ satisfies Assumption \ref{as:phi} with $\ls = 0$ and arbitrary $\Delta > 0$.
        \item[(b)] If $\lambda \geq 3$, then $\phi_\lambda$ satisfies Assumption \ref{as:phi} with $\ls = e (\lambda - 1)(\lambda - 2)$ and $\Delta = 1 / (\lambda - 2)$.
        \item[(c)] If $\lambda < 3$ and $\lambda \notin \{1, 2\}$, then $\phi_\lambda$ satisfies Assumption \ref{as:phi} with $\ls = e^2 |(\lambda - 1)(\lambda - 2)|$ and $\Delta = 1 / (5 - \lambda)$.
    \end{itemize}
\end{Prop}
The proof of Proposition \ref{prop:power} is deferred to Appendix \ref{app:prop_power_proof}.
We are ready to formulate the main result of the present paper.
\begin{Th}
	\label{th:phi-div}
	Let Assumption \ref{as:phi} be satisfied and let $p_{\min} = \min\limits_{1 \leq j \leq r} p_j > 0$.
	Assume that $n$ satisfies
	\[
		5 |\phi'''(1)| \left( p_j(1 - p_j) + \log n \right) \leq 4 \phi''(1) \sqrt{n} p_j \quad \text{for all } j \in \{1, \dots, r\},
	\]
	\[
	    5\log n \leq 2 p_{\min} \Delta \sqrt n,
	\]
	and
	\[
		16r^3 + 16r^2 \log n \leq np_{\min}.
	\]
	Then, for any $\delta \in (0, 1)$, with probability at least $1 - \delta$ (over $\btheta$ uniformly distributed on the 	unit sphere $\cs^{n-1}$), it holds that
	\begin{equation}
	    \label{eq:rand_phi-div_bound}
		\sup\limits_{t \in \R} \left| \p\left( \ct_\phi > t \cond \right)  - \p\left( Z > t 
		\right) \right|
		\lesssim \left(\frac{\phi'''(1)}{\phi''(1)}\right)^2 \frac{r^{3/2} + (\log n)^{3/2}}{n p_{\min}} + \frac{C_{r-1} r \log^2(1 / \delta)}{n p_{\min}} + \frac{\ls \sqrt{r} (\log n)^4}{\phi''(1) n p_{\min}^3}.
	\end{equation}
	Here $Z \sim \chi^2(r-1)$ and $C_{r-1}$ is the same constant as in \eqref{eq:au20}. 
\end{Th}

Proposition \ref{prop:power} allows us to specify the result of Theorem \ref{th:phi-div} for the case of power divergence test statistics.
\begin{Co}
    Let $p_{\min} = \min\limits_{1 \leq j \leq r} p_j > 0$. For any $\lambda \in \R$, define $\Delta_\lambda$ and $\ls_\lambda$ as follows:
    \[
        \Delta_\lambda
        = \begin{cases}
            +\infty, \quad \text{if $\lambda = 1$, or $\lambda = 2$}\\
            1 / (\lambda - 2), \quad \text{if $\lambda \geq 3$,}\\
            1/ (5 -\lambda), \quad \text{otherwise,}
        \end{cases}
        \qquad\text{and}\qquad
        \ls_\lambda
        = \begin{cases}
            0, \quad \text{if $\lambda = 1$, or $\lambda = 2$}\\
            e(\lambda - 1)(\lambda - 2), \quad \text{if $\lambda \geq 3$,}\\
            e^2|(\lambda - 1)(\lambda - 2)|, \quad \text{otherwise.}
        \end{cases}
    \]
    Assume that $n$ satisfies
	\[
		5 |\lambda - 1| \left( p_j(1 - p_j) + \log n \right) \leq 4 \sqrt{n} p_j \quad \text{for all } j \in \{1, \dots, r\},
	\]
	\[
	    5\log n \leq 2 p_{\min} \Delta_\lambda \sqrt n,
	\]
	and
	\[
		16r^3 + 16r^2 \log n \leq np_{\min}.
	\]
	Then, for any $\delta \in (0, 1)$, with probability at least $1 - \delta$ (over $\btheta$ uniformly distributed on the 	unit sphere $\cs^{n-1}$), it holds that
	\begin{equation*}
	    \sup\limits_{t \in \R} \left| \p\left( \ct_{\phi_\lambda} > t \cond \right)  - \p\left( Z > t \right) \right|
		\lesssim \frac{(\lambda - 1)^2 (r^{3/2} + (\log n)^{3/2})}{n p_{\min}} + \frac{C_{r-1} r \log^2(1 / \delta)}{n p_{\min}} + \frac{\ls_\lambda \sqrt{r} (\log n)^4}{n p_{\min}^3},
	\end{equation*}
	where $\phi_\lambda$ is defined in \eqref{eq:power_div}, $Z \sim \chi^2(r-1)$, and $C_{r-1}$ is the same constant as in \eqref{eq:au20}.
\end{Co}

\begin{Rem}
In the case of randomized Pearson's statistic
\[
    \ct_{\phi_1} = \sum\limits_{j=1}^r \frac{(X_j^\btheta)^2}{p_j},
\]
the first and the last term in the right hand side of \eqref{eq:rand_phi-div_bound} vanish, and we obtain that, with probability at least $1 - \delta$,
\[
    \sup\limits_{t \in \R} \left| \p\left( \sum\limits_{j=1}^r \frac{(X_j^\btheta)^2}{p_j} > t \cond \right)  - \p\left( Z > t \right) \right|
	\lesssim \frac{C_{r-1} r \log^2(1 / \delta)}{n p_{\min}},
\]
where $Z \sim \chi^2(r-1)$. We conjecture that the $\log n$ factors can be removed in a general case. However, it may require a different technique.
\end{Rem}

It is worth mentioning that, in practice, a statistician often has an access to the aggregated data $\by$, rather than to the actual outcomes $\bfeta_1, \dots, \bfeta_n$. In this case, the statistic $\ct_\phi$ cannot be computed directly. We use the following trick to overcome this issue. Given $\by$, define $\widetilde\bfeta_1 = \widetilde\bfeta_2 = \ldots = \widetilde\bfeta_{Y_1} = (1, 0, \dots, 0)^\top$, $\widetilde\bfeta_{Y_1 + 1} = \widetilde\bfeta_{Y_1 + 2} = \ldots = \widetilde\bfeta_{Y_1 + Y_2} = (0, 1, 0, \dots, 0)^\top$, etc. Obviously, there exists a permutation $\bpi = (\pi(1), \dots, \pi(n))^\top$ (depending on $\by$), such that $\widetilde\bfeta_i = \bfeta_{\pi(i)}$ for all $i \in \{1, \dots, n\}$. Let $\btheta$ be uniformly distributed on the unit sphere $\cs^{n-1}$,
define
\[
    \smash{\widetilde\bx}^\btheta
    = \sum\limits_{i=1}^n \theta_i \left(\widetilde\bfeta_i - \bp \right)
\]
and compute a randomized phi-divergence test statistic based on $\widetilde\bfeta_1, \dots, \widetilde\bfeta_n$ (which requires the knowledge of $\by$ only):
\[
	\widetilde\ct_\phi = \frac{2n}{\phi''(1)} \sum\limits_{j=1}^r p_j \phi\left( 1 + 
	\frac{\smash{\widetilde X}_j^\btheta}{\sqrt{n} 
	p_j} \right).
\]
Note that, conditionally on $\by$,
\[
    \left( \sum\limits_{i=1}^n \theta_i \left(\widetilde\bfeta_i - \bp \right) \,\Big\vert\, \by \right)
    = \left( \sum\limits_{i=1}^n \theta_i \left(\bfeta_{\pi(i)} - \bp \right) \,\Big\vert\, \by \right)
    = \left( \sum\limits_{i=1}^n \theta_{\pi^{-1}(i)} \left(\bfeta_i - \bp \right) \,\Big\vert\, \by \right)
    \eqd \left( \sum\limits_{i=1}^n \theta_i \left(\bfeta_i - \bp \right) \,\Big\vert\, \by \right),
\]
where $\bpi^{-1}$ is an inverse permutation for $\bpi$, and the sign ``$\eqd$'' stands for the equality in distribution. Equality of the conditional distributions yields
\[
    \sum\limits_{i=1}^n \theta_i \left(\widetilde\bfeta_i - \bp \right)
    \eqd \sum\limits_{i=1}^n \theta_i \left(\bfeta_i - \bp \right),
\]
and, hence, $\widetilde\ct_\phi$ has the same distribution as $\ct_\phi$ given by \eqref{eq:rand_phi-div}.

We defer the proof of Theorem \ref{th:phi-div} to the next section and split it into several steps for the sake of readability. The proof goes as follows. Using Taylor's expansion, it is easy to show that the randomized test statistic $\ct_\phi$ can be represented in the form
\[
	\ct_\phi
	= \sum\limits_{j=1}^r \frac{(X_j^\btheta)^2}{p_j} + \frac{\phi'''(1)}{3\phi''(1)} \sum\limits_{j=1}^r \frac{(X_j^\btheta)^3}{\sqrt n p_j^2} + \rr,
\]
where $\rr$ is a term of order $O((\log n)^4 / n)$.
The main ingredient of the proof is a careful analysis of
\[
    \sum\limits_{j=1}^r \frac{(X_j^\btheta)^2}{p_j} + \frac{\phi'''(1)}{3\phi''(1)} \sum\limits_{j=1}^r \frac{(X_j^\btheta)^3}{\sqrt n p_j^2}
\]
with cubic terms of order $O(n^{-1/2})$, which did not appear when we considered the randomized rank-based statistic $\ct_\J$.
Though it is clear that, due to Theorem \ref{th:au20}, the Kolmogorov distance between the law of the quadratic part $\sum_{j=1}^r (X_j^\btheta)^2 / p_j$ and $\chi^2(r-1)$ decays as fast as $O(1/n)$, it is a challenge to show that addition of the cubic summands of order $O(n^{-1/2})$ does not contaminate this rate of convergence. We address this issue in Lemma \ref{lem:gauss_approx} and Lemma \ref{lem:chi_squared_approx} below.

\section{Proofs}
\label{sec:proofs}

This section contains proofs of Theorem \ref{th:rankings} and Theorem \ref{th:phi-div}.

\subsection{Proof of Theorem \ref{th:rankings}}
\label{sec:proof_rankings}

\noindent{\bf Step 1.}\quad
In \cite[p. 314]{sen68}, the author claims that
\[
	\E \bV_1^\J = \overline\J \bone,
	\qquad
	\text{Var}(\bV_1^\J) = \sigma^2_\J I_r - \frac{\sigma^2_\J}{r} \bone \bone^\top \equiv \Sigma_\J. 
\]
Let us take a closer look at the covariance matrix $\Sigma_\J$. It is easy to check that $\Sigma_\J \bone = \bzero$ and that
\[
	\Sigma_\J \bu = \sigma^2_\J \bu
\]
for any vector $\bu$ which is orthogonal to $\bone$. Hence, $\Sigma_\J$ has one eigenvector corresponding to $0$ and $(r-1)$ orthogonal eigenvectors corresponding to $\sigma^2_\J$. This yields nice properties of the Moore-Penrose pseudoinverse $\Sigma_\J^\dag$ of the matrix $\Sigma_\J$. Namely, $\Sigma_\J^\dag \bone = \bzero$ and, for any vector $\bu$ such that $\bu^\top \bone = 0$, we have
\[
	\Sigma_\J^\dag \bu = \frac{1}{\sigma^2_\J} \bu.
\]
Taking into account that $\bone^\top \bV_i^\J = r\overline\J$ for all $i$ from $1$ to $n$ and, therefore,
\[
	\left( \bV_i^\J - \overline\J \bone\right)^\top \bone = 0
	\quad \text{almost surely,}
\]
we obtain that $\F$ can be rewritten in the following form:
\[
		\ct_\J = \left\| \sum\limits_{i=1}^n \theta_i \left(\Sigma_\J^\dag\right)^{1/2} \left( \bV_i^\J - \E\bV_i^\J \right) \right\|^2.
\]
Since $\bV_1^\J - \E\bV_1^\J, \dots, \bV_n^\J - \E\bV_n^\J$ are orthogonal to $\bone$ almost surely, we restrict our attention on the orthogonal complement of $\bone$. Applying Theorem \ref{th:au20}, we get
\begin{align*}
	&
	\sup\limits_{t \in \R} \left| \p\left( \ct_\J > t \cond \right)  - \p\left( Z > t \right) \right|
	\\&
	= \sup\limits_{t \in \R} \left| \p\left( \sum\limits_{i=1}^n \theta_i \left(\Sigma_\J^\dag\right)^{1/2} \left( \bV_i^\J - \E\bV_i^\J \right) \in \B(0, t) \cond \right)  - \p\left( \widetilde\bV \in \B(0, t) \right) \right|
	\\&
	\leq \frac{C_{r-1} \E \left( \left( \bV_1^\J - \E\bV_1^\J \right)^\top \Sigma_\J^\dag \left( \bV_1^\J - \E\bV_1^\J \right) \right)^2 \log^2(1 / \delta)}{n},
\end{align*}
where $Z \sim \chi^2(r-1)$, $\widetilde\bV \sim \N(\bzero, I_{r-1})$.

\noindent{\bf Step 2.}\quad
It remains to bound the expectation
\[
	\E \left( \left( \bV_1^\J - \E\bV_1^\J \right)^\top \Sigma_\J^\dag \left( \bV_1^\J - \E\bV_1^\J \right) \right)^2
\]
to finish the proof of the theorem.

\begin{Lem}
	\label{lem:rankings_moment}
	With the notations introduced above, it holds that
	\[
		\E \left( \left( \bV_1^\J - \E\bV_1^\J \right)^\top \Sigma_\J^\dag \left( \bV_1^\J - \E\bV_1^\J \right) \right)^2
		\leq \frac{r^2 B^2}{\sigma_\J^2}.
	\]
\end{Lem}

\begin{proof}[Proof of Lemma \ref{lem:rankings_moment}]
	A reader can recall that
	\[
		\Sigma_\J^\dag \bu = \frac{1}{\sigma_\J^2} \bu
	\]
	for any vector $\bu$ which is orthogonal to $\bone$.
	Thus,
	\[
		\left( \bV_1^\J - \E\bV_1^\J \right)^\top \Sigma_\J^\dag \left( \bV_1^\J - \E\bV_1^\J \right)
		= \frac{1}{\sigma_\J^2} \left\| \bV_1^\J - \E\bV_1^\J \right\|^2.
	\]
	Then
	\[
		\E \left( \left( \bV_1^\J - \E\bV_1^\J \right)^\top \Sigma_\J^\dag \left( \bV_1^\J - \E\bV_1^\J \right) \right)^2
		= \frac{1}{\sigma_\J^4} \E \left\| \bV_1^\J - \E\bV_1^\J \right\|^4.
	\]
	Finally, using the fact that the absolute value of each component of $\bV_1^\J - \E\bV_1^\J$ does not exceed $B$ almost surely, we obtain
	\begin{align*}
		\E \left\| \bV_1^\J - \E\bV_1^\J \right\|^4
		\leq r B^2 \E \left\| \bV_1^\J - \E\bV_1^\J \right\|^2
		= r B^2 \text{Tr}(\Sigma_\J)
		= r^2 B^2 \sigma_\J^2.
	\end{align*}
	Hence,
	\begin{equation}
		\label{eq:moment_bound}
		\E \left( \left( \bV_1^\J - \E\bV_1^\J \right)^\top \Sigma_\J^\dag \left( \bV_1^\J - \E\bV_1^\J \right) \right)^2
		\leq \frac{r^2 B^2}{\sigma_\J^2},
	\end{equation}
	and the proof is finished.

\end{proof}

\subsection{Proof of Theorem \ref{th:phi-div}}
\label{sec:proof_phi-div}

\noindent{\bf Step 1.}\quad
Taylor's expansion yields
\begin{align}
	\label{eq:phi-div_taylor}
	\phi\left(1 + \frac{X_j^\btheta}{\sqrt n p_j} \right)
	= \frac{\phi''(1) (X_j^\btheta)^2}{2 n p_j^2}
	+ \frac1{2} \int\limits_0^1 \phi''' \left(1 + \frac{v X_j^\btheta}{\sqrt n p_j}\right) \frac{(X_j^\btheta)^3}{n\sqrt n p_j^3}(1 - v)^2 \dd v.
\end{align}
Then one may rewrite the randomized test statistic \eqref{eq:rand_phi-div} in the following form:
\begin{align*}
	\ct_\phi
	&
	= \underbrace{\sum\limits_{j=1}^r \frac{(X_j^\btheta)^2}{p_j} + \frac{\phi'''(1)}{3\phi''(1)} \sum\limits_{j=1}^r 			\frac{(X_j^\btheta)^3}{\sqrt n p_j^2}}_{\cq({\bx^\btheta})}
	\\&\quad
	+ \underbrace{\frac1{\phi''(1)} \sum\limits_{j=1}^r \int\limits_0^1 
	\left[ \phi''' \left(1 + \frac{v X_j^\btheta}{\sqrt n p_j}\right) - \phi'''(1) \right] \frac{(X_j^\btheta)^3}{\sqrt n p_j^2} (1 - v)^2 \dd v}_{\rr}.
\end{align*}
Here we introduced a function $\cq : \R^r \rightarrow \R$
\begin{equation}
	\label{eq:q}
	\cq({\bf x}) = \sum\limits_{j=1}^r \frac{x_j^2}{p_j}
	+ \frac{\phi'''(1)}{3\phi''(1)} \sum\limits_{j=1}^r \frac{x_j^3}{\sqrt n p_j^2}.
\end{equation}

\medskip

\noindent{\bf Step 2.}\quad
We continue with an upper bound on the absolute value of the remainder term $\rr$. Recall that, due to Assumption \ref{as:phi}, the third derivative of $\phi$ is $\ls$-Lipschitz on $[1 - \Delta, 1 + \Delta]$.
\begin{Lem}
	\label{lem:r}
	There is an event $E_1$, $\p(E_1 \,\vert\, \btheta) \geq 1 - 2r/n$, such that, conditionally on $\btheta$, 
	\[
		\left|X_j^\btheta \right|
		\leq \frac{5p_j(1 - p_j)}4 + \frac{5\log n}4
	\]
	simultaneously for all $j \in \{1, \dots, r\}$ on $E_1$.
	Furthermore, on this event, one has
	\[
		\left|\rr \right| \leq \frac{2\ls}{\phi''(1) n} + \frac{2\ls r (\log n)^4}{\phi''(1) n p_{\min}^3}.
	\]
\end{Lem}

\begin{proof}[Proof of Lemma \ref{lem:r}.]
Bernstein's inequality and the union bound yield that there is an event $E_1$ such that
$\p(E_1 \,\vert\, \btheta) \geq 1 - 2r/n$ and, conditionally on $\btheta$, 
\begin{align*}
	\left|X_j^\btheta \right|
	&
	\leq \sqrt{2\text{Var}(X_j^\theta) \log n} + \frac{2\log n}3
	\\&
	= \sqrt{2p_j(1 - p_j) \log n} + \frac{2\log n}3
	\\&
	\leq \frac{5p_j(1 - p_j)}4 + \frac{5\log n}4
\end{align*}
simultaneously for all $j \in \{1, \dots, r\}$ on $E_1$.
According to Assumption \ref{as:phi}, $\phi'''$ is Lipschitz in a vicinity of $1$.
Since on $E_1$
\[
    \frac{\left|X_j^\btheta \right|}{\sqrt n p_j}
    \leq \frac5{4\sqrt n} \left(1 - p_j + \frac{\log n}{p_j} \right)
    \leq \frac5{4\sqrt n} \left(1 + \frac{\log n}{p_{\min}} \right)
    \leq \frac{5\log n}{2 p_{\min} \sqrt n}
    \leq \Delta
\]
for all $j \in \{1, \dots, r\}$, the absolute value of the remainder term $\rr$ does not exceed
\begin{align*}
	\left|\rr \right|
	&
	\leq \frac1{\phi''(1)} \sum\limits_{j=1}^r \int\limits_0^1 \left| \phi''' \left(1 + \frac{v X_j^\btheta}{\sqrt n p_j}\right) - \phi'''(1) \right| \frac{|X_j^\btheta|^3}{\sqrt n p_j^2} (1 - v)^2 \dd v
	\\&
	\leq \frac{\ls}{\phi''(1)} \sum\limits_{j=1}^r \frac{(X_j^\btheta)^4}{n p_j^3} 
	\int\limits_0^1 v (1 - v)^2 dv
	= \frac{\ls}{12 \phi''(1)} \sum\limits_{j=1}^r \frac{(X_j^\btheta)^4}{n p_j^3}.
\end{align*}
Furthermore, the following holds on $E_1$:
\begin{align*}
	\left|\rr \right|
	&
	\leq \frac{\ls}{12 \phi''(1)} \sum\limits_{j=1}^r \frac{5^4(p_j(1 - p_j) + 
	\log n)^4}{4^4 n p_j^3}
	\\&
	\leq \frac{\ls}{12 \phi''(1)} \sum\limits_{j=1}^r \frac{5^4}{32 n p_j^3} \left( p_j^4 (1 - p_j)^4 + (\log n)^4 \right).
\end{align*}
Here we used the inequality $(a + b)^4 \leq 8(a^4 + b^4)$ which holds for any $a, b \geq 
0$.
Hence, we have
\[
    \left|\rr \right|
	\leq \frac{2\ls}{\phi''(1) n} \sum\limits_{j=1}^r \frac{p_j^4 + (\log n)^4}{p_j^3}
	\leq \frac{2\ls}{\phi''(1) n} + \frac{2\ls r (\log n)^4}{\phi''(1) n p_{\min}^3}.
\]
\end{proof}

\medskip

\noindent{\bf Step 3.}\quad
Let us fix $t \in \R$ and consider $\p\left( \ct_\phi > t \,\vert\, \btheta \right)$.
The next lemma shows that the probability $\p\left( \ct_\phi > t \,\vert\, \btheta \right)$ changes not too 
much if one replaces $\ct_\phi$ by a polynomial of a Gaussian 
random vector.

\begin{Lem}
	\label{lem:gauss_approx}
	Assume that the sample size $n$ satisfies
	\[
		\frac{5 |\phi'''(1)|}4 \left( p_j(1 - p_j) + \log n \right) \leq \phi''(1) \sqrt{n} p_j
	\]
	for all $j \in \{1, \dots, r\}$.
	Let $\bxtilde \sim \N(0, \Sigma)$ be a Gaussian random vector in $\R^r$ with $\Sigma$ defined in \eqref{eq:sigma}. For any $t \in \R$, it holds that
	\begin{align*}
		\p\left( \ct_\phi > t \cond \right)
		&
		\geq \p\left( \cq(\bxtilde) > t + \frac{2\ls}{\phi''(1) n} + \frac{2\ls r (\log n)^4}{\phi''(1) n p_{\min}^3} \right) - \frac{2r}n
		\\&\quad
		- \frac{C_{r-1} \E ((\bfeta_1 - \bp)^\top \Sigma^\dag (\bfeta_1 - \bp))^2 \log^2(1 / \delta)}{n} 
	\end{align*}
	and
	\begin{align*}
		\p\left( \ct_\phi > t \cond \right)
		&
		\leq \p\left( \cq(\bxtilde) > t - \frac{2\ls}{\phi''(1) n} - \frac{2\ls r (\log n)^4}{\phi''(1) n p_{\min}^3} \right) + \frac{2r}n
		\\&\quad
		+ \frac{C_{r-1} \E ((\bfeta_1 - \bp)^\top \Sigma^\dag (\bfeta_1 - \bp))^2 \log^2(1 / \delta)}{n},
	\end{align*}
	where $\cq$ is given by \eqref{eq:q}.
\end{Lem}

\begin{proof}[Proof of Lemma \ref{lem:gauss_approx}.]

We start with the proof of the lower bound.
Note that
\begin{align*}
	\p\left( \ct_\phi > t \cond \right)
	&
	\geq \p\left( \left\{\cq(\bx^\btheta) + \rr> t \right\} \cap E_1 \cond \right)
	\\&
	\geq \p\left( \left\{\cq(\bx^\btheta) > t + \frac{2\ls}{\phi''(1) n} + \frac{2\ls r (\log n)^4}{\phi''(1) n p_{\min}^3} \right\} \cap E_1 \cond \right).
\end{align*}
Direct calculation shows that, if $\phi'''(1) \neq 0$, then the Hessian of $\cq({\bf x})$ is 
positive definite on the set
\[
	\bigotimes\limits_{j=1}^r \left( -\frac{\phi''(1) \sqrt{n} p_j}{|\phi'''(1)|}, \frac{\phi''(1) 
	\sqrt{n} p_j}{|\phi'''(1)|} \right).
\]
Otherwise, the Hessian of $f(x)$ is positive definite everywhere.
Hence, if
\[
	\frac{5 |\phi'''(1)|}4 \left( p_j(1 - p_j) + \log n \right) \leq \phi''(1) \sqrt{n} p_j
\]
for all $j \in \{1, \dots, r\}$, then
\[
	\Big\{ {\bf x} \in \R^r : \cq({\bf x}) > t \text{ and } |x_j| \leq 5(p_j(1 - p_j) + 
	\log n) / 4 \text{ for all $j$} \Big\}
\]
is a convex Borel set.
We are going to apply Theorem \ref{th:au20} to the weighted sum
\[
	\bx^\btheta = \sum\limits_{i=1}^n \theta_i (\bfeta_i - \bp).
\]
Note that $(\bfeta_1 - \bp)^\top\bone = 0$ almost surely and that the covariance matrix $\Sigma$ has rank $(r-1)$.
This means that $\bfeta_1$ is supported on a $(r-1)$-dimensional subspace $\text{Im}(\Sigma)$. 
From now on, we restrict ourselves on this subspace.
Applying Theorem \ref{th:au20} to the weighted sum $\sum\limits_{i=1}^n \theta_i (\Sigma^\dag)^{1/2}(\bfeta_i - \bp)$, where $\Sigma^\dag$ is the Moore-Penrose pseudo-inverse of $\Sigma$, we obtain that
\begin{align*}
	&
	\sup\limits_{B \in \mathfrak B(\text{Im}(\Sigma))} \left| \p\left( 
	\sum\limits_{i=1}^n \theta_i (\Sigma^\dag)^{1/2}(\bfeta_i - \bp) \in B \,\Big\vert\, \btheta \right) - \p( \mathbf{V} \in 	B ) \right|
	\\&\notag
	\leq \frac{C_{r-1} \E ((\bfeta_1 - \bp)^\top \Sigma^\dag (\bfeta_1 - \bp))^2 \log^2(1 / \delta)}{n}
\end{align*}
with probability at least $1 - \delta$ over the external randomization.
Here $\mathbf{V} \sim \N(0, I_{r-1})$ and $\mathfrak B(\text{Im}(\Sigma))$ is the class of Borel convex sets in $\text{Im}(\Sigma)$.
Since $(\bfeta - \bp) \in \text{Im}(\Sigma)$ almost surely, it holds that $\Sigma^{1/2} (\Sigma^\dag)^{1/2} (\bfeta - \bp) = (\bfeta - \bp)$.
This yields that, with probability at least $1 - \delta$ over the external randomization,  
\begin{equation}
	\label{eq:weighted_multinomial_clt}
	\sup\limits_{B \in \mathfrak B(\text{Im}(\Sigma))} \left| \p\left( \bx^\btheta \in B \,\Big\vert\, \btheta \right) - \p( \bxtilde \in B ) \right|
	\leq \frac{C_{r-1} \E ((\bfeta_1 - \bp)^\top \Sigma^\dag (\bfeta_1 - \bp))^2 \log^2(1 / \delta)}{n},
\end{equation}
where $\bxtilde \sim \N(0, \Sigma)$ is a Gaussian random vector.
In its turn, \eqref{eq:weighted_multinomial_clt} implies
\begin{align*}
	\p\left( \ct_\phi > t \cond \right)
	&
	\geq \p\left( \left\{\cq(\bx^\btheta) > t + \frac{2\ls}{\phi''(1) n} + \frac{2\ls r (\log n)^4}{\phi''(1) n p_{\min}^3} \right\} \cap E_1 \cond \right)
	\\&
	\geq \p\left( \cq(\bxtilde) > t +\frac{2\ls}{\phi''(1) n} + \frac{2\ls r (\log n)^4}{\phi''(1) n p_{\min}^3} \text{ and } |\widetilde X_j| \leq 		5(p_j(1 - p_j) + 	\log n) / 4 \text{ for all $j$} \right)
	\\&\quad
	- \frac{C_{r-1} \E ((\bfeta_1 - \bp)^\top \Sigma^\dag (\bfeta_1 - \bp))^2 \log^2(1 / \delta)}{n}.
\end{align*}
Finally, Hoeffding's inequality and the union bound imply that, with probability at least $1 
- 2r/n$, 
\[
	|\widetilde X_j| \leq \sqrt{2 p_j(1 - p_j) \log n}
	< \frac54 \Big(p_j(1 - p_j) + \log n \Big)
\]
simultaneously for all $j \in \{1, \dots, r\}$. This yields
\begin{align*}
	&
	\p\left( \cq(\bxtilde) > t +\frac{2\ls}{\phi''(1) n}
	+ \frac{2\ls r (\log n)^4}{\phi''(1) n p_{\min}^3} \text{ and } |\widetilde X_j|
	\leq 5(p_j(1 - p_j) + \log n) / 4 \text{ for all $j$} \right)
	\\&
	\geq \p\left( \cq(\bxtilde) > t + \frac{2\ls}{\phi''(1) n} + \frac{2\ls r (\log n)^4}{\phi''(1) n p_{\min}^3} \right) - \frac{2r}n
\end{align*}
and
\begin{align*}
	\p\left( \ct_\phi > t \cond \right)
	&
	\geq \p\left( \cq(\bxtilde) > t + \frac{2\ls}{\phi''(1) n} + \frac{2\ls r (\log n)^4}{\phi''(1) n p_{\min}^3} \right) - \frac{2r}n
	\\&\quad
	- \frac{C_{r-1} \E ((\bfeta_1 - \bp)^\top \Sigma^\dag (\bfeta_1 - \bp))^2 \log^2(1 / \delta)}{n}.
\end{align*}
Similarly, we deduce
\begin{align*}
	\p\left( \ct_\phi > t \cond \right)
	&
	\leq \p\left( \left\{\cq(\bx^\btheta) > t - \frac{2\ls}{\phi''(1) n} - \frac{2\ls r (\log n)^4}{\phi''(1) n p_{\min}^3} \right\} \cap E_1 \cond \right) + \p(E_1 \,\vert\, \btheta)
	\\&
	\leq \p\left( \cq(\bxtilde) > t - \frac{2\ls}{\phi''(1) n} - \frac{2\ls r (\log n)^4}{\phi''(1) n p_{\min}^3} \right) + \frac{2r}n
	\\&\quad
	+ \frac{C_r \E ((\bfeta_1 - \bp)^\top \Sigma^\dag (\bfeta_1 - \bp))^2 \log^2(1 / \delta)}{n}.
\end{align*}
\end{proof}

\medskip

\noindent{\bf Step 4.}\quad
The next lemma provides an upper bound on the expectation $\E ((\bfeta_1 - \bp)^\top \Sigma^\dag (\bfeta_1 - \bp))^2$.
\begin{Lem}
	\label{lem:pseudo-inverse}
	It holds that
	\[
		\E ((\bfeta_1 - \bp)^\top \Sigma^\dag (\bfeta_1 - \bp))^2 \leq \sum\limits_{j=1}^r p_j^{-1} \leq \frac{r}{p_{\min}}.
	\]
\end{Lem}

\begin{proof}
	Show that $\Sigma^{1/2} D^{-1} \Sigma^{1/2}$ is a projector of onto the orthogonal complement of $\bone = (1, 	\dots, 1)^\top$.
	It is straightforward to check that $\Sigma \bone = \bzero$ and $\Sigma$ is matrix of rank $(r-1)$.
	Hence, the image of $\Sigma^{1/2} D^{-1} \Sigma^{1/2}$ is a $(r-1)$-dimensional subspace which is orthogonal to 	$\bone$.
	Moreover, $\Sigma^{1/2} D^{-1} \Sigma^{1/2}$ is symmetric and idempotent:
	\begin{align*}
		\Sigma^{1/2} D^{-1} \Sigma D^{-1} \Sigma^{1/2}
		&
		= \Sigma^{1/2} D^{-1} (D - \bp\bp^T) D^{-1} \Sigma^{1/2}
		\\&
		= \Sigma^{1/2} (D^{-1} - D^{-1}\bp\bp^T D^{-1}) \Sigma^{1/2}
		\\&
		= \Sigma^{1/2} (D^{-1} - \bone\bone^T) \Sigma^{1/2}
		\\&
		= \Sigma^{1/2} D^{-1} \Sigma^{1/2},
	\end{align*}
	and, hence, $\Sigma^{1/2} D^{-1} \Sigma^{1/2}$ is a projector of onto the orthogonal complement of $\bone = (1, \dots, 1)^\top$.
	This implies that, for any vector $\bu$, $\bu^\top \bone = 0$, one has
	\[
		\|\bu\|^2 = \bu^\top \Sigma^{1/2} D^{-1} \Sigma^{1/2} \bu.
	\]
	In particular,
	\begin{align*}
		(\bfeta_1 - p)^\top \Sigma^\dag (\bfeta_1 - p)
		&
		= (\bfeta_1 - p)^\top (\Sigma^\dag)^{1/2} \Sigma^{1/2} D^{-1} \Sigma^{1/2} (\Sigma^\dag)^{1/2} (\bfeta_1 - p)
		\\&
		= (\bfeta_1 - p)^\top D^{-1} (\bfeta_1 - p).
	\end{align*}
	Then
	\begin{align*}
		\E \left((\bfeta_1 - \bp)^\top \Sigma^\dag (\bfeta_1 - \bp)\right)^2
		&
		= \E \left((\bfeta_1 - \bp)^\top D^{-1} (\bfeta_1 - \bp)\right)^2
		\\&
		= \sum\limits_{j=1}^r \left((\be_j - \bp)^\top D^{-1} (\be_j - \bp) \right)^2 p_j,
	\end{align*}
	where $\be_1, \dots, \be_r$ is the standard basis in $\R^r$, that is, for any $j \in \{1, \dots, r\}$, the $j$-th 			component of $\be_j$ is equal to one and all the other components are equal to zero.
	Note that, for any $j \in \{1, \dots, r\}$, it holds that
	\begin{align*}
	    0
	    &
	    \leq (\be_j - \bp)^\top D^{-1} (\be_j - \bp)
		\\&
		= \be_j^\top D^{-1} \be_j - 2 \bp^\top D^{-1} \be_j + \bp^\top D^{-1} \bp
		\\&
		= p_j^{-1} - 2 \bone^\top \be_j + \bone^\top \bp
		\\&
		= p_j^{-1} - 1
		< p_j^{-1}.
	\end{align*}
	This yields
	\[
		\sum\limits_{j=1}^r \left((\be_j - \bp)^\top D^{-1} (\be_j - \bp) \right)^2
		\leq \sum\limits_{j=1}^r p_j^{-1}
		\leq \frac{r}{p_{\min}}.
	\]
\end{proof}

\medskip

\noindent{\bf Step 5.}\quad
Substituting the bound of Lemma \ref{lem:pseudo-inverse} into the results of Lemma \ref{lem:gauss_approx}, we obtain that, with probability at least $1 - \delta$ (over the external randomization)
\begin{equation}
	\label{eq:lem_pseudo-inverse_implication1}
	\p\left( \ct_\phi > t \cond \right)
	\geq \p\left( \cq(\bxtilde) > t + \frac{2\ls}{\phi''(1) n} + \frac{2\ls r (\log n)^4}{\phi''(1) n p_{\min}^3} \right) - \frac{2r}n - \frac{C_{r-1} r \log^2(1 / \delta)}{n p_{\min}} 
\end{equation}
and
\begin{equation}
	\label{eq:lem_pseudo-inverse_implication2}
	\p\left( \ct_\phi > t \cond \right)
	\leq \p\left( \cq(\bxtilde) > t - \frac{2\ls}{\phi''(1) n} - \frac{2\ls r (\log n)^4}{\phi''(1) n p_{\min}^3} \right) + \frac{2r}n + \frac{C_{r-1} r \log^2(1 / \delta)}{n p_{\min}},
\end{equation}
where $\bxtilde \sim \N(0, \Sigma)$.
Our next goal is to show that the distribution of $\cq(\bxtilde)$ is close to $\chi^2(r-1)$.
Introduce $\bz = D^{-1/2} \widetilde \bx \sim \N(0, D^{-1/2} \Sigma D^{-1/2})$ and $\bq = D^{-1/2} \bp$.
It is easy to check that $D^{-1/2} \Sigma D^{-1/2}$ is a projector of rank $(r - 1)$.
Indeed, it holds that
\begin{align*}
    \left(D^{-1/2} \Sigma D^{-1/2})\right)^2
    &
    = \left( I_r - \left(D^{-1/2}\bp\right) \left(D^{-1/2}\bp\right)^\top \right)^2
    = \left( I_r - \bq \bq^\top \right)^2
    \\&
    = I_r - \left(2 - \bq^\top \bq\right) \bq\bq^\top
    = I_r - \bq\bq^\top
    = D^{-1/2} \Sigma D^{-1/2},
\end{align*}
where we used the equality
\[
    \bq^\top \bq
    = \bp^\top D^{-1} \bp
    = \bp^\top \bone
    = 1.
\]
Consequently, the distribution of $\|\bz\|^2$ is exactly $\chi^2(r-1)$.
On the other hand, $\cq(\widetilde\bx)$ can be rewritten as
\[
	\cq(\widetilde\bx)
	= \sum\limits_{j=1}^r \frac{\widetilde X_j^2}{p_j}
	+ \frac{\phi'''(1)}{3\phi''(1)} \sum\limits_{j=1}^r \frac{\widetilde X_j^3}{\sqrt n p_j^2}
	= \sum\limits_{j=1}^r \z_j^2
	+ \frac{\phi'''(1)}{3\phi''(1)} \sum\limits_{j=1}^r \frac{\z_j^3}{\sqrt{n p_j}}.
\]

\begin{Lem}
	\label{lem:chi_squared_approx}
	Assume that
	\[
		16r^3 + 16r^2 \log n \leq np_{\min}.
	\]
	Then it holds that
	\[
		\sup\limits_{t \in \R} \left| \p\left( \|\bz\|^2 + \frac{\phi'''(1)}{3\phi''(1)} \sum\limits_{j=1}^r \frac{\z_j^3}{\sqrt{n p_j}} > t \right) - \p\left( Z > t \right) \right|
		\lesssim \left( \frac{\phi'''(1)}{\phi''(1)}\right)^2 \frac{r^{3/2} + (\log n)^{3/2}}{n p_{\min}},
	\]
	where $Z \sim \chi^2(r-1)$.
\end{Lem}

\begin{proof}
	Let $\rho = \|\bz\|$, $\boldsymbol{\tau} = \bz / \|\bz\|$, and introduce
	\[
		\cs = \frac{\phi'''(1)}{3\phi''(1)} \sum\limits_{j=1}^r \frac{\tau_j^3 \sqrt{p_{\min}}}{\sqrt{p_j}}.
	\]
	Then it is enough to show that
	\[
		\sup\limits_{t \in \R} \left| \p\left( \rho^2 + \frac{\rho^3 \cs}{\sqrt{n p_{\min}}} > t \right) - \p\left( \rho^2 > t \right) \right|
		\lesssim \left(\frac{\phi'''(1)}{\phi''(1)}\right)^2 \frac{r^{3/2} + (\log n)^{3/2}}{n p_{\min}}.
	\]
	If $\phi'''(1) = 0$, then the statement is trivial.
	In the rest of the proof, we assume that $\phi'''(1) > 0$ without loss of generality.
	The case $\phi'''(1) < 0$ is absolutely similar.
	
	\medskip
	
	\noindent{\bfseries\itshape Case 1: $t \leq 0$.}\quad
	If $t \leq 0$, then
	\begin{align*}
		\p\left( \rho^2 + \frac{\rho^3 \cs}{\sqrt{n p_{\min}}} \leq t \right)
		&
		\leq \p\left( \rho^2 + \frac{\rho^3 \cs}{\sqrt{n p_{\min}}} \leq 0 \right)
		= \p\left( 1 + \frac{\rho \cs}{\sqrt{n p_{\min}}} \leq 0 \right)
		\\&
		\leq \p\left( \rho \leq -\frac{3\phi''(1)}{\phi'''(1)}\sqrt{n p_{\min}} \right)
		\lesssim \left( \frac{\phi'''(1)}{\phi''(1)}\right)^2 \frac{1}{n p_{\min}}.
	\end{align*}
	Here we took into account that, by the definition of $\cs$,
	\[
		|\cs|
		\leq \frac{\phi'''(1)}{3\phi''(1)} \sum\limits_{j=1}^r \frac{|\tau_j|^3 \sqrt{p_{\min}}}{\sqrt{p_j}}
		\leq \frac{\phi'''(1)}{3\phi''(1)} \sum\limits_{j=1}^r \tau_j^2
		\leq  				\frac{\phi'''(1)}{3\phi''(1)}.
	\]
	\noindent{\bfseries\itshape Case 2: $t > 8(r - 1) + 8\log n$.}
	In this case,
	\[
		\p\left( \rho^2 > t \right)
		\leq \frac{\E e^{\rho^2 / 4}}{e^{t / 4}}
		= 2^{(r-1) / 2} e^{-t / 4}
		\leq \frac1n
	\]
	and
	\begin{align*}
		\p\left( \rho^2 + \frac{\rho^3 \cs}{\sqrt{n p_{\min}}} > t \right)
		&
		\leq \p\left( \rho^2 > \frac t2 \right) + \p\left( \frac{\rho^3 \cs}{\sqrt{n p_{\min}}} > \frac t2 \right)
		\\&
		\leq \p\left( \rho^2 > \frac t2 \right) + \p\left( \rho^3 > \frac{3\phi''(1) t}{2\phi'''(1)} \sqrt{n p_{\min}} \right)
		\\&
		\leq 2^{(r - 1) / 2} e^{-t/8} + 2^{(r - 1) / 2} \exp\left\{-\frac14 \left(\frac{3\phi''(1) t}{2\phi'''(1)}\right)^{2/3} \left(n p_{\min}\right)^{1/3}\right\}
		\\&
		\leq \frac1n + 2^{(r - 1) / 2} \exp\left\{-\frac{r^{1/3}}4 \left(\frac{3\phi''(1) t}{2\phi'''(1)}\right)^{2/3} \left(\frac{n p_{\min}}r\right)^{1/3}\right\}
		\\&
		\lesssim \left( \frac{\phi'''(1)}{\phi''(1)}\right)^2 \frac r{n p_{\min}}.
	\end{align*}
	\noindent{\bfseries\itshape Case 3: $0 < t \leq 8(r - 1) + 8\log n$.}\quad
	It remains to bound
	\[
		\sup\limits_{0 < t \leq 8(r - 1) + 8\log n} \left| \p\left( \rho^2 + \frac{\rho^3 \cs }{\sqrt{n p_{\min}}} > t \right) - \p\left( \rho^2 > t \right) \right|.
	\]
	For any $t \in (0, 8(r - 1) + 8\log n)$, let a random variable $\sigma_t$ be a root of the equation
	\[
		t = \sigma_t^2 + \frac{\sigma_t^3 \cs}{\sqrt{n p_{\min}}}
	\]
	from the interval $(\sqrt{0.5t}, \sqrt{2t})$.
	Such a root exists, because
	\[
		\frac{t}2 + \frac{t\cs \sqrt{t}}{2 \sqrt{2np_{\min}}}
		\leq \frac{t}2 \left( 1 + \frac{ \sqrt{t}}{\sqrt{2np_{\min}}} \right)
		\leq \frac{t}2 \left( 1 + \frac{2 \sqrt{(r-1) + \log n}}{\sqrt{np_{\min}}} \right)
		\leq t
	\]
	and
	\[
		2t + \frac{2t\cs \sqrt{2t}}{\sqrt{np_{\min}}}
		\geq 2t \left( 1 - \frac{ \sqrt{2t}}{\sqrt{np_{\min}}} \right)
		\geq 2t \left( 1 - \frac{4\sqrt{(r-1) + \log n}}{\sqrt{np_{\min}}} \right)
		\geq t.
	\]
	Note that $\sigma_t$ is independent of $\rho$.
	Then
	\begin{align*}
		&
		\sup\limits_{0 < t \leq 8(r - 1) + 8\log n} \left| \p\left( \rho^2 + \frac{\rho^3 \cs }{\sqrt{n p_{\min}}} > t \right) - \p\left( \rho^2 > t \right) \right|
		\\&
		= \sup\limits_{0 < t \leq 8(r - 1) + 8\log n} \left| \E \p\left( \rho^2 + \frac{\rho^3 \cs}{\sqrt{n p_{\min}}} > t \,\Big\vert\, \cs \right) - \p\left( \rho^2 > t \right) \right|
		\\&
		= \sup\limits_{0 < t \leq 8(r - 1) + 8\log n} \left| \E \p\left( \rho^2 + \frac{\rho^3 \cs}{\sqrt{n p_{\min}}} > \sigma_t^2 + \frac{\sigma_t^3 \cs}{\sqrt{n p_{\min}}} \,\Big\vert\, \cs \right) - \E\p\left( \rho^2 > \sigma_t^2 + \frac{\sigma_t^3 \cs}{\sqrt{n p_{\min}}} \,\Big\vert\, \cs \right) \right|.
	\end{align*}
	We need to solve the inequality
	\begin{equation}
	    \label{eq:cubic_inequality}
	    \rho^2 + \frac{\rho^3 \cs}{\sqrt{n p_{\min}}} > \sigma_t^2 + \frac{\sigma_t^3 \cs}{\sqrt{n p_{\min}}},
	    \qquad \rho > 0,
	\end{equation}
	with respect to $\rho$.
	For this purpose, we use the representation
	\[
		\frac{\rho^3 \cs}{\sqrt{n p_{\min}}} + \rho^2 - \frac{\sigma_t^3 \cs}{\sqrt{n p_{\min}}} - \sigma_t^2
		= \frac{\cs}{\sqrt{n p_{\min}}} \left( \rho - \sigma_t \right) \left( \rho^2 + \rho\sigma_t + \sigma_t^2 + 	\frac{\sqrt{np_{\min}}}{\cs} (\rho + \sigma_t) \right),
	\]
	which helps us to find all the roots of the cubic polynomial in the left hand side.
	If $\cs > 0$, then the map $x \mapsto x^2 + x^3 \cs / \sqrt{n p_{\min}}$ is monotone, and \eqref{eq:cubic_inequality} is equivalent to $\rho > \sigma_t$. 
	Hence,
	\[
		\p\left( \rho^2 + \frac{\rho^3 \cs}{\sqrt{n p_{\min}}} > \sigma_t^2 + \frac{\sigma_t^3 \cs}{\sqrt{n p_{\min}}} \,\Big\vert\, \cs > 0\right)
		= \p\left(\rho > \sigma_t \,\Big\vert\, \cs > 0\right). 	
	\]
	Otherwise, if $\cs < 0$, then the solution of \eqref{eq:cubic_inequality} is
	\[
	    \sigma_t < \rho <  \frac12 \left( \frac{\sqrt{n p_{\min}}}{|\cs|} - \sigma_t \right) \left( 1 + \sqrt{1 + 
		\frac{4\sigma_t}{\frac{\sqrt{n p_{\min}}}{|\cs|} - \sigma_t}} \right).
	\]
	Therefore,
	\begin{align*}
		&
		\p\left( \rho^2 + \frac{\rho^3 \cs}{\sqrt{n p_{\min}}} > \sigma_t^2 + \frac{\sigma_t^3 \cs}{\sqrt{n p_{\min}}} \,\Big\vert\, \cs < 0\right)
		\\&
		= \p\left(\sigma_t < \rho <  \frac12 \left( \frac{\sqrt{n p_{\min}}}{|\cs|} - \sigma_t \right) \left( 1 + \sqrt{1 + 
		\frac{4\sigma_t}{\frac{\sqrt{n p_{\min}}}{|\cs|} - \sigma_t}} \right) \,\Big\vert\, \cs < 0 \right)
		\\&
		= \p\left( \rho > \sigma_t \,\big\vert\, \cs < 0 \right) + O\left( \exp\left\{- O\left( \left(\frac{\phi''(1)}{\phi'''(1)}\right)^2 n p_{\min} \right) \right\}\right)
		\\&
		= \p\left( \rho > \sigma_t \,\big\vert\, \cs < 0 \right) + O\left( \left( \frac{\phi'''(1)}{\phi''(1)}\right)^2 \frac 1{n p_{\min}} \right).
	\end{align*}
	Let $\ps(x)$ be the density of $\rho$.
	Then
	\begin{align*}
		&
		\left| \p\left( \rho^2 + \frac{\rho^3 \cs}{\sqrt{n p_{\min}}} > \sigma_t^2 + \frac{\sigma_t^3 \cs}{\sqrt{n p_{\min}}} \right) - \p\left( \rho^2 > \sigma_t^2 + \frac{\sigma_t^3 \cs}{\sqrt{n p_{\min}}} \right) \right|
		\\&
		= \left| \p\left( \rho > \sigma_t \right) - \p\left( \rho^2 > \sigma_t^2 + \frac{\sigma_t^3 \cs}{\sqrt{n p_{\min}}} \right) \right| + O\left( \left( \frac{\phi'''(1)}{\phi''(1)}\right)^2 \frac 1{n p_{\min}} \right)
		\\&
		= \left| \E \int\limits_{\sigma_t}^{\sigma_t \sqrt{1 + \frac{\sigma_t \cs}{\sqrt{n p_{\min}}}} } \ps(x) \dd x \right| + 
		O\left( \left( \frac{\phi'''(1)}{\phi''(1)}\right)^2 \frac 1{n p_{\min}} \right)
		\\&
		= \left|\int\limits_{\E \sigma_t}^{\E \sigma_t \sqrt{1 + \frac{\sigma_t \cs}{\sqrt{n p_{\min}}}}} \ps(x) \dd x \right| + 
		O\left( \left( \frac{\phi'''(1)}{\phi''(1)}\right)^2 \frac 1{n p_{\min}} \right).
	\end{align*}
	Finally, since
	\[
		|\cs|
		\leq \frac{\phi'''(1)}{3\phi''(1)} \sum\limits_{j=1}^r \frac{|\tau_j|^3 \sqrt{p_{\min}}}{\sqrt{r p_j}}
		\leq \frac{\phi'''(1)}{3\phi''(1)} \sum\limits_{j=1}^r \tau_j^2
		\leq \frac{\phi'''(1)}{3\phi''(1)},
	\]
	and
	\[
	    |\sigma_t| \leq \sqrt{2t} \lesssim \sqrt{r} + \sqrt{\log n} \quad \text{almost surely},
	\]
	it holds that
	\[
		\E \sigma_t \sqrt{1 + \frac{\sigma_t \cs}{\sqrt{n p_{\min}}}}
		= \E \sigma_t + \E \frac{\sigma_t^2 \cs}{2\sqrt{n p_{\min}}} + O\left( \left( \frac{\phi'''(1)}{\phi''(1)}\right)^2 \frac{r^{3/2} + (\log n)^{3/2}}{n p_{\min}} \right)
	\]
	and
	\[
		\E \frac{\sigma_t^2 \cs}{2\sqrt{n p_{\min}}}
		= \E \frac{t \cs}{2\sqrt{n p_{\min}}} - \E \frac{\sigma_t^3 \cs^2}{2 n p_{\min}}
		= O\left( \left( \frac{\phi'''(1)}{\phi''(1)}\right)^2 \frac{r^{3/2} + (\log n)^{3/2}}{n p_{\min}} \right).
	\]
	This and the fact that $\max\limits_{x > 0} \ps(x) \lesssim 1$ yield the desired result.
	
\end{proof}

\medskip

\noindent{\bf Step 6.}\quad
Let $Z \sim \chi^2(r-1)$.
Lemma \ref{lem:chi_squared_approx} and the inequalitites \eqref{eq:lem_pseudo-inverse_implication1} and \eqref{eq:lem_pseudo-inverse_implication2} yield that, for any $t \in \R$, 
\begin{align*}
	\p\left( \ct_\phi > t \cond \right)
	&
	\geq \p\left( Z > t + \frac{2\ls}{\phi''(1) n} + \frac{2\ls r (\log n)^4}{\phi''(1) n p_{\min}^3} \right)
	\\&\quad
	- \frac{2r}n - 
	\left( \frac{\phi'''(1)}{\phi''(1)}\right)^2 \frac{r^{3/2} + (\log n)^{3/2}}{n p_{\min}}
	- \frac{C_{r-1} r \log^2(1 / \delta)}{n p_{\min}},
\end{align*}
and
\begin{align*}
	\p\left( \ct_\phi > t \cond \right)
	&
	\leq \p\left( Z > t - \frac{2\ls}{\phi''(1) n} - \frac{2\ls r (\log n)^4}{\phi''(1) n p_{\min}^3} \right)
	\\&\quad
	+ \frac{2r}{n} + \left( \frac{\phi'''(1)}{\phi''(1)}\right)^2 \frac{r^{3/2} + (\log n)^{3/2}}{n p_{\min}}
	+ \frac{C_{r-1} r \log^2(1 / \delta)}{n p_{\min}}.
\end{align*}
To finish the proof of Theorem \ref{th:phi-div}, we use the anti-concentration property of 
a chi-squared random variable.
According to \cite[Theorem 2.7]{gnsu19}, it holds that
\[
	\left| \p\left( Z > t + \frac{2\ls}{\phi''(1) n} + \frac{2\ls r (\log n)^4}{\phi''(1) n p_{\min}^3} \right) - \p\left( Z > t \right) 
	\right|
	\lesssim \frac{1}{\sqrt{r}} \left( \frac{\ls}{\phi''(1) n} + \frac{\ls r (\log n)^4}{\phi''(1) n p_{\min}^3} \right)
\]
and
\[
	\left| \p\left( Z > t - \frac{2\ls}{\phi''(1) n} - \frac{2\ls r (\log n)^4}{\phi''(1) n p_{\min}^3} \right) - \p\left( Z > t \right) 
	\right|
	\lesssim \frac{1}{\sqrt{r}} \left( \frac{\ls}{\phi''(1) n} + \frac{\ls r (\log n)^4}{\phi''(1) n p_{\min}^3} \right).
\]
Hence,
\begin{align*}
	\sup\limits_{t \in \R} \left| \p\left( \ct_\phi > t \cond \right)  - \p\left( Z > t 
	\right) \right|
	&
	\lesssim \left(\frac{\phi'''(1)}{\phi''(1)}\right)^2 \frac{r^{3/2} + (\log n)^{3/2}}{n p_{\min}}
	\\&\quad
	+ \frac{C_{r-1} r \log^2(1 / \delta)}{n p_{\min}}
	+ \frac{\ls \sqrt{r} (\log n)^4}{\phi''(1) n p_{\min}^3}.
\end{align*}

\appendix

\section{Proof of Proposition \ref{prop:rankings}}
\label{app:prop_proof}

The proof of the proposition relies on \citep[Chapter 7, Theorem 1]{esseen45} and \citep[Theorem 1.2]{gz14}.
Let
\[
	\Sigma_\J = \sigma_\J^2 I_r - \frac{\sigma_\J^2}r \bone \bone^\top.
\]
Recall that, according to \citep[p. 314]{sen68},
\[
	\E \bV_1^\J = \overline\J \bone
	\quad\text{and}\quad
	\text{Var}(\bV_1^\J) = \Sigma_\J. 
\]
Repeating the same argument as in the proof of Theorem \ref{th:rankings}, we obtain that
\[
	\ct_\J = \left\| \sum\limits_{i=1}^n \left( \Sigma_\J^\dag \right)^{1/2} \left( \bV_i^\J - \E \bV_i^\J \right) \right\|^2,
\]
where the vectors
\[
	\bW_i^\J = \left( \Sigma_\J^\dag \right)^{1/2} \left( \bV_i^\J - \E \bV_i^\J \right), \quad 1 \leq i \leq n,
\]
are i.i.d. random vectors supported on a proper subspace of $\R^r$ of dimension $r-1$, which is orthogonal to $\bone = (1, \dots, 1)^\top$. Choose an orthonormal basis $\bu_1, \dots, \bu_{r-1}$ in this subspace and put
\[
	U = \left( \bu_1, \dots, \bu_{r-1} \right) \in \R^{r \times (r-1)}.
\]
It is clear that $U^\top$ performs an orthogonal embedding of $\bW_1^\J, \dots, \bW_n^\J$ into $\R^{r-1}$. This yields that
\[
	\ct_\J
	= \left\| \sum\limits_{i=1}^n \bW_i^\J \right\|^2
	= \left\| \sum\limits_{i=1}^n U^\top \bW_i^\J \right\|^2. 
\]
Moreover,
\[
	\text{Var}\left(U^\top \bW_1^\J\right) = U^\top \text{Var}\left(\bW_1^\J\right) U = I_{r-1}.
\]
Next, the cases $2 \leq r \leq 5$ and $r \geq 6$ are considered separately.

\medskip

\noindent{\it Case 1: $2 \leq r \leq 5$.}\quad
In this case, we use \citep[Chapter 7, Theorem 1]{esseen45}, which yields that
\begin{equation}
	\label{eq:esseen45}
	\sup\limits_{t \in \R} \left| \p\left( \ct_\J > t \right) - \p( Z > t ) \right|
	\leq \frac{c_r}{n^{1 - 1 / r}} \left( \sum\limits_{j=1}^{r-1} \E \left(\bu_j^\top \bW_1^\J \right)^4 \right)^{3/2},
\end{equation}
where $Z \sim \chi^2(r-1)$ and $c_r$ is a constant depending on $r$ only. Let us elaborate on
\[
	\sum\limits_{j=1}^{r-1} \E \left(\bu_j^\top \bW_1^\J \right)^4.
\]
Note that $\Sigma_\J$ is such that $\Sigma_\J \bone = 0$ and $\Sigma_\J \bu = \sigma_\J^2 \bu$ for any $\bu$ which is orthogonal to $\bone$. Then
\[
	\left(\Sigma_\J^\dag \right)^{1/2} \bone = 0
	\quad\text{and}\quad
	\left(\Sigma_\J^\dag \right)^{1/2} \bu = \frac1{\sigma_\J} \bu \quad
	\text{for all $\bu$ which is orthogonal to $\bone$}.
\]
Thus, for any $j \in \{1, \dots, r-1\}$, it holds that
\[
	\E \left(\bu_j^\top \bW_1^\J \right)^4
	= \E \left(\bu_j^\top \left(\Sigma_\J^\dag \right)^{1/2} \left( \bV_1^\J - \E \bV_1^\J \right) \right)^4
	= \frac1{\sigma_\J^4} \E \left(\bu_j^\top \left( \bV_1^\J - \E \bV_1^\J \right) \right)^4.
\]
Since $|\J(k) - \overline\J | \leq B$ for all $k \in \{1, \dots, r\}$, we have
\begin{align*}
	\E \left(\bu_j^\top \left( \bV_1^\J - \E \bV_1^\J \right) \right)^4
	&
	\leq \E \|\bu_j\|^2 \left\| \bV_1^\J - \E \bV_1^\J \right\|^2 \left(\bu_j^\top \left( \bV_1^\J - \E \bV_1^\J \right) \right)^2
	\\&
	\leq B^2 r \E \left(\bu_j^\top \left( \bV_1^\J - \E \bV_1^\J \right) \right)^2
	\\&
	= B^2 r \bu_j^\top \Sigma_\J \bu_j
	= B^2 \sigma_\J^2 r.
\end{align*}
Hence,
\begin{equation}
	\label{eq:esseen45_moment}
	\sum\limits_{j=1}^{r-1} \E \left(\bu_j^\top \bW_1^\J \right)^4
	\leq \frac{B^2 r^2}{\sigma_\J^2}.
\end{equation}

\medskip

\noindent{\it Case 2: $r \geq 6$.}\quad
Applying \citep[Theorem 1.2]{gz14} to the vectors $U^\top \bW_1^\J, \dots, U^\top \bW_n^\J$, we obtain that
\begin{equation}
	\label{eq:gz14}
	\sup\limits_{t \in \R} \left| \p\left( \ct_\J > t \right) - \p( Z > t ) \right|
	\leq \frac{c_r \E \| U^\top \bW_1^\J \|^4}{n},
\end{equation}
where $Z \sim \chi^2(r-1)$ and the constant $c_r$ depends on $r$ only. An upper bound on $\E \| U^\top \bW_1^\J \|^4$ follows from Lemma \ref{lem:rankings_moment}:
\begin{equation}
	\label{eq:gz14_moment}
	\E \| U^\top \bW_1^\J \|^4
	= \E \| \bW_1^\J \|^4
	= \E \left( \left( \bV_1^\J - \E\bV_1^\J \right)^\top \Sigma_\J^\dag \left( \bV_1^\J - \E\bV_1^\J \right) \right)^2
	\leq \frac{r^2 B^2}{\sigma_\J^2}.
\end{equation}
Hence, taking \eqref{eq:esseen45}, \eqref{eq:esseen45_moment}, \eqref{eq:gz14}, and \eqref{eq:gz14_moment} together, we conclude that
\[
	\sup\limits_{t \in \R} \left| \p\left( \ct_\J > t \right) - \p( Z > t ) \right| =
	\begin{cases}
		O\left( B^3 n^{ -1 + 1/r} / \sigma_\J^3 \right), \quad \text{if $2 \leq r \leq 5$},\\
		O_r\left( B^2 n^{-1} / \sigma_\J^2 \right), \quad \text{if $r \geq 6$},
	\end{cases}
\]
where $Z \sim \chi^2(r-1)$.

\section{Proof of Proposition \ref{prop:power}}
\label{app:prop_power_proof}

Note that, by the definition of $\phi_\lambda$, $\phi_\lambda(1) = \phi_\lambda'(1) = 0$ and $\phi_\lambda''(1) = 1$ for any $\lambda \in \R$. It remains to specify the constants $\ls$ and $\Delta$.

\medskip

\noindent{\bfseries\itshape Proof of Proposition \ref{prop:power}a.}\quad
Note that
\[
    \phi_1'''(u) \equiv 0
    \quad \text{and} \quad
    \phi_2'''(u) \equiv 1.
\]
Obviously, these functions satisfy Assumption \ref{as:phi} with $\ls = 0$ and any $\Delta > 0$.

\medskip

\noindent{\bfseries\itshape Proof of Proposition \ref{prop:power}b.}\quad
Direct calculations show that
\[
    \frac{\dd^4 \phi_\lambda(u)}{\dd u^4}
    = (\lambda - 1)(\lambda - 2) u^{\lambda - 3},
    \quad \text{for all $\lambda \in \R$.}
\]
Consider the case $\lambda \geq 3$.
Take $\Delta = 1 / (\lambda - 2)$. Then, for any $u \in [1 - \Delta, 1 + \Delta]$, it holds that
\begin{align*}
    \left| \frac{\dd^4 \phi_\lambda(u)}{\dd u^4} \right|
    &
    \leq (\lambda - 1)(\lambda - 2) (1 + \Delta)^{\lambda - 3}
    \leq (\lambda - 1)(\lambda - 2) e^{\Delta(\lambda - 3)}
    \\&
    = (\lambda - 1)(\lambda - 2) e^{(\lambda - 3) / (\lambda - 2)}
    < e (\lambda - 1)(\lambda - 2),
\end{align*}
where the last inequality is due to the fact that $0 \leq (\lambda - 3) / (\lambda - 2) < 1$ for all $\lambda \geq 3$.

\medskip

\noindent{\bfseries\itshape Proof of Proposition \ref{prop:power}c.}\quad
Recall that, for any $\lambda \in \R$,
\[
    \frac{\dd^4 \phi_\lambda(u)}{\dd u^4}
    = (\lambda - 1)(\lambda - 2) u^{\lambda - 3}.
\]
If $\lambda < 3$, take $\Delta = 1 / (5 - \lambda) \in (0, 1/2]$ and note that, for any $u \in [1 - \Delta, 1 + \Delta]$,
\[
    \left| \frac{\dd^4 \phi_\lambda(u)}{\dd u^4} \right|
    \leq \left| (\lambda - 1)(\lambda - 2) \right| (1 - \Delta)^{\lambda - 3}.
\]
Using Taylor's expansion with the integral remainder term, we obtain that
\begin{align*}
    -\log(1 - \Delta)
    &
    = \Delta + \frac{\Delta^2}2 \int\limits_0^1 \left(1 - s\Delta\right)^{-2} \dd s
    \\&
    \leq \Delta \left( 1 + \frac{\Delta}{2 (1 - \Delta)^2} \right).
\end{align*}
Substituting $\Delta = 1 / (5 - \lambda)$, we get
\[
    \frac{\Delta}{(1 - \Delta)^2}
    = \frac{5 - \lambda}{4 - \lambda} \cdot \frac{1}{4 - \lambda}
    < 2 \cdot 1 = 2.
\]
Here we used the inequalities
\[
    \frac{5 - \lambda}{4 - \lambda} < 2
    \quad \text{and} \quad
    \frac{1}{4 - \lambda} < 1,
    \quad \text{for all $\lambda < 3$}.
\]
Hence, $1 / (1 - \Delta) \leq e^{2\Delta}$ and
\begin{align*}
    \frac{\dd^4 \phi_\lambda(u)}{\dd u^4}
    &
    \leq \left|(\lambda - 1)(\lambda - 2)\right| e^{2 \Delta (3 - \lambda)}
    \leq \left|(\lambda - 1)(\lambda - 2)\right| e^{2 \Delta (3 - \lambda)}
    \\&
    = \left|(\lambda - 1)(\lambda - 2)\right| e^{2 (3 - \lambda) / (5 - \lambda)}
    < e^2 \left|(\lambda - 1)(\lambda - 2)\right|,
\end{align*}
where the last inequality is due to the fact that $0 \leq (3 - \lambda) / (5 - \lambda) < 1$ for all $\lambda < 3$. The proof of Proposition \ref{prop:power} is finished.

\bigskip

\begin{acks}[Acknowledgments]
The article was prepared within the framework of the HSE University Basic Research Program. The first author is a Young Russian Mathematics award winner and would like to thank its sponsors and jury. The second author carried out the research within the scope of the Moscow Center for Fundamental and Applied Mathematics at Moscow State University. The authors are grateful to the associate editor and two anonymous referees for the careful reading of the paper and for their valuable and constructive remarks that improved the quality of this work.
\end{acks}

\bibliographystyle{imsart-number}
\bibliography{references}

\end{document}